\documentclass[12pt,a4paper,reqno]{amsproc}
\usepackage{amsmath}
\usepackage{amssymb}
\usepackage{amsthm}
\usepackage[backrefs]{amsrefs}
\usepackage[german,danish,english]{babel}
\usepackage[utf8]{inputenc} %æøå m.m.
\usepackage[T1]{fontenc}
\usepackage{graphicx}
\usepackage{color}
\usepackage{mathrsfs}
\usepackage{comment}
\usepackage{multirow}
\usepackage{enumitem}
\usepackage[footnote,draft,danish,silent]{fixme}
\usepackage{MnSymbol}
\usepackage[all]{xy}

\numberwithin{equation}{section}

\newcommand{\NN}{\mathbb N}
\newcommand{\ZZ}{\mathbb Z}

\newcommand{\del}{\subseteq}
\newcommand{\eps}{\varepsilon}
\newcommand{\ph}{\varphi}

\newcommand{\gclass}{\mathcal{C}}

\DeclareMathOperator{\homdist}{HomDist}

\newcommand{\frob}[1]{\lVert #1 \rVert_{\mathrm{Frob}}}
\newcommand{\opnorm}[1]{\lVert #1 \rVert_{\mathrm{op}}}
\newcommand{\hsnorm}[1]{\lVert #1 \rVert_{\mathrm{HS}}}

\newcommand{\Uni}{\mathrm{U}}
\newcommand{\uf}{\mathcal{U}}

\newcommand{\Uniuf}{\Uni_\uf^{\lVert\cdot\rVert}}
\newcommand{\Matuf}{\mathrm{M}_\uf^{\lVert\cdot\rVert}}

\DeclareMathOperator{\dist}{dist}
\DeclareMathOperator{\spam}{span}

\DeclareMathOperator{\Ra}{Re}

\DeclareMathOperator{\diag}{diag}

\DeclareMathOperator{\id}{id}

\DeclareMathOperator{\Tr}{Tr}

\DeclareMathOperator{\Hom}{Hom}
\DeclareMathOperator{\defect}{def}

\newtheorem{stn}{Sætning}[section]
\newtheorem{theorem}[stn]{Theorem}
\newtheorem{cor}[stn]{Corollary}
\newtheorem{lemma}[stn]{Lemma}
\newtheorem{prop}[stn]{Proposition}
\newtheorem{question}[stn]{Question}
\theoremstyle{definition}\newtheorem{defi}[stn]{Definition}
\theoremstyle{remark}\newtheorem{rem}[stn]{Remark}

\def\cH{{\mathcal H}}

\def\cS{{\mathcal S}}

\def\cC{{\mathcal C}}

\newcommand{\Z}{\mathbb Z}

\def\ls#1{\langle #1 \rangle}

\title[Asymptotic representations]{Stability, cohomology vanishing, and non-approximable groups}

\author[M.\ De Chiffre]{Marcus De Chiffre}
\address{M.D.C., TU Dresden, Germany}
\email{marcus\_dorph.de\_chiffre@tu-dresden.de}

\author[L.\ Glebsky]{Lev Glebsky}
\address{L.G., Universidad Aut\'onoma de San Luis Potos\'\i, M\'exico.}
\email{glebsky@cactus.iico.uaslp.mx}

\author[A.\ Lubotzky]{Alexander Lubotzky}
\address{A.L., Hebrew University, Israel}
\email{alexlub@math.huji.ac.il}

\author[A.\ Thom]{Andreas Thom}
\address{A.T., TU Dresden, Germany}
\email{andreas.thom@tu-dresden.de}

\begin{document}
\begin{abstract}
Several well-known open questions (such as: are all groups sofic/hyperlinear?) have a common form: can all groups be approximated by asymptotic homomorphisms into the symmetric groups $\mathrm{Sym}(n)$ (in the sofic case) or the finite dimensional unitary groups $\Uni(n)$ (in the hyperlinear case)? In the case of $\Uni(n)$, the question can be asked with respect to different metrics and norms. 
This paper answers, for the first time, one of these versions, showing that there exist fintely presented groups which are \emph{not} approximated by $\Uni(n)$ with respect to the Frobenius norm $\frob{T}=\sqrt{\sum_{i,j=1}^n|T_{ij}|^2},T=[T_{ij}]_{i,j=1}^n\in\mathrm{M}_n(\mathbb C)$.
Our strategy is to show that some higher dimensional cohomology vanishing phenomena implies \emph{stability}, that is, every
Frobenius-approximate homomorphism into finite-dimensional unitary groups is close to an actual homomorphism.
This is combined with existence results of certain non-residually finite central extensions of lattices in some simple $p$-adic Lie groups. 
These groups act on high rank Bruhat-Tits buildings and satisfy the needed vanishing cohomology phenomenon and are thus stable and not Frobenius-approximated.
\end{abstract}

\maketitle
%\tableofcontents

\section{Introduction}
Since the very beginning of the study of groups, groups were studied by looking at their orthogonal and unitary representations. It is very natural to relax the notion of a representation and require the group multiplication to be preserved only up to little mistakes in a suitable metric. First variations of this topic appeared already in the work of Turing \cite{turing1938finite} and later Ulam \cite[Chapter VI]{ulam}. This theme knows many variations, ranging from sofic approximations as introduced by Gromov \cite{gromov} and operator-norm approximations that appeared in the theory of operator algebras \cite{MR1437044, cde} to questions related to Connes' Embedding Problem, see \cite{MR0454659, MR2460675} for details. In each case, approximation properties of groups are studied relative to a particular class of metric groups. More specifically, let $\Gamma$ be a countable group and let $(G_n,d_n)_{n=1}^{\infty}$ be a sequence of metric groups with bi-invariant metrics $d_n$. We say that $\Gamma$ is $(G_n,d_n)_{n=1}^{\infty}$-approximated, if there exists a separating sequence of asymptotic homomorphisms $\ph_n\colon \Gamma \to G_n$, i.e.\ a sequence of maps $\ph_n$ that becomes multiplicative in the sense that
$$\lim_{n \to \infty} d_n(\ph_n(gh),\ph_n(g)\ph_n(h)) = 0, \qquad\text{for all\ } g,h\in \Gamma,$$
which is also separating, that is, $d_{n}(\ph_n(g),1_{G_n})$ is bounded away from zero for all $g \neq 1_{\Gamma}$, see Section \ref{apphom} for precise definitions. Several examples of this situation have been studied in the literature (see \cite{Asurvey} for a survey):
\begin{enumerate}
\item[(i)] $G_n = {\rm Sym}(n)$, the symmetric group on an $n$-point set, with $d_n$ the normalized Hamming distance. In this case, $(G_n,d_n)_{n=1}^{\infty}$-approximated groups are called \emph{sofic}, see \cite{gromov, MR2460675}.
\item[(ii)] $G_n$  an arbitrary finite group equipped with any bi-invariant metric. In this case, approximated groups are called \emph{weakly sofic}, or ${\mathcal C}$-\emph{approximated} depending on a particular restricted family $\mathcal C$ of finite groups. An interesting connection to pro-finite group theory and recent advances can be found in \cite{glebskyrivera2008sofic, chnikthom}.
\item[(iii)] $G_n={\rm U}(n)$, the unitary group on an $n$-dimensional Hilbert space, where the metric $d_n$ is induced by the normalized Hilbert-Schmidt norm $\hsnorm T = \sqrt{n^{-1} \sum_{i,j=1}^n |T_{ij}|^2}$. In this case, approximated groups are sometimes called \emph{hyperlinear} \cite{MR2460675}. 
%and the question whether all groups are hyperlinear is closely related to Connes' Embedding Problem, \cite{MR0454659, MR2460675}.
\item[(iv)] $G_n={\rm U}(n)$, where the metric $d_n$ is induced by the operator norm $\|T\|_{\rm op} = \sup_{\|v\|=1} \|Tv\|.$ In this case, groups which are $(G_n,d_n)_{n=1}^{\infty}$-approximated groups are called {\rm MF}, see \cite{cde}.
\item[(v)] $G_n={\rm U}(n)$, where the metric $d_n$ is induced by the \emph{unnormalized} Hilbert-Schmidt norm $\|T\|_{\rm Frob}= \sqrt{ \sum_{i,j=1}^n |T_{ij}|^2}$, also called Frobenius norm. We will speak about Frobenius-approximated groups in this context.
\end{enumerate}

Note that the approximation properties are \emph{local} in the sense that only finitely many group elements and their relations have to be considered for fixed $\ph_n$. This is in stark contrast to the \emph{uniform} situation, which -- starting with the work of of Grove-Karcher-Ruh and Kazhdan \cite{johnson, art:kazh} -- is much better understood, see \cite{MR3038548, chozth}.

Well-known and longstanding problems, albeit in different fields of mathematics, ask if \emph{any} group exists which is not approximated in either of the above settings. In setting (i), this is Gromov's question whether all groups are sofic \cite{gromov, MR2460675}. The similar question in the context of (iii) is closely related to Connes' Embedding Problem \cite{MR0454659, MR2460675}. Indeed, the existence of a non-hyperlinear group, whould answer Connes' Embedding Problem in the negative. In \cite{MR1437044}, Kirchberg asked whether any stably finite $C^*$-algebra is embeddable into an norm-ultraproduct of matrix algebras, implying a positive answer to the approximation problem in the sense of (iv) for any group. Recent breakthrough results imply that any amenable group is MF, i.e.\ approximated in the  sense of (iv), see \cite{MR3583354}.%\fxnote{I changed this a little so that it does not claim that K. conjectured anything}

\vspace{0.1cm}

In this paper, we want to introduce a conceptually new technique that allows us to provide groups that are not approximated in the sense of (v) above, i.e.\ we show that there are finitely presented groups which are not approximated by unitary groups $\{{\rm U}(n) \mid n \in \mathbb N\}$ with their Frobenius norm. Our techniques do not apply directly to the context of (iii), so we cannot say anything conclusive about Connes' Embedding Problem, but since the norms in (iii) and (v) are related by a normalization constant, we believe that we provide a promising new angle of attack.%\fxnote{I also changed this a little to my taste. Feel free to disagree}

\vspace{0.1cm}

Before we start out explaining our strategy and some notation let us state the main results of this article.

\begin{theorem} \label{main}
There exist finitely presented groups which are not Frobenius-approximated. 
%the universal central extension of a particular lattice in ${\rm Sp}(2n,\mathbb Q_p)$ 
\end{theorem}

The groups we construct are central extensions of cocompact lattices in simple $p$-adic Lie groups. Specifically, we can take certain central extensions of ${\rm U}(2n) \cap {\rm Sp}(2n,\mathbb Z[i,1/p])$ for $n\geq 3$ and $p$ a large enough prime.

To prove Theorem \ref{main}, we use the notion of stability: A group is called $(G_n,d_n)$\emph{-stable} if every asymptotic homomorphism (not necessarily a separating one) is ``close'' to a true homomorphism (see Definition \ref{stabdef}). Now, if $G_n=\mathrm{U}(n)$ and $\Gamma$ is $(G_n,d_n)$-approximated and $(G_n,d_n)$-stable, one easily deduces that $\Gamma$ must be residually finite.
This basic observation suggests a way to find non-approximated groups: find a group $\Gamma$ which is stable but not residually finite. This method has failed so far for two reasons: (1) it is very difficult to prove stability directly and (2) even in the case where stability was proven,
(see e.g.\ \cite{MR3350728} and the references therein as well as \cites{becklub})
it was proven in a way that completely classifies asymptotic homomorphism and it is shown that all are close to a genuine homomorphisms. Thus, only groups which are already approximated have been shown to be stable so far.
The main technical novelty of our paper is the following theorem which provides a sufficient condition for a group to be Frobenius-approximated without assuming a priori that the group is approximated.

\begin{theorem} \label{main2}
Let $\Gamma$ be a finitely presented group such that \[H^2(\Gamma,\mathcal H_\pi)=\{0\}\] for every unitary representation $\pi \colon \Gamma \to {\rm U}(\mathcal H_\pi)$. Then, any asymptotic homomorphism $\ph_n \colon \Gamma \to {\rm U}(n)$ w.r.t.\ the Frobenius norm is asymptotically close to a sequence of homomorphisms, i.e. $\Gamma$ is Frobenius-stable.
\end{theorem}

The appearance of vanishing second cohomology groups may look surprising at first sight, but, inf fact, one can translate the question of approximating an asymptotic homomorphism by a true homomorphism to a question about splitting an exact sequence. When the norm is submultiplicative (as is the case of the Frobenius norm but \emph{not} of the normalized Hilbert-Schmidt norm) the kernel of this splitting problem is abelian (see Section \ref{altsec}). It is well-known that vanishing of the second cohomology with abelian coefficients means splitting of suitable exact sequences and hence is relevant the question of stability.
It is also interesting to observe that the second cohomology has already appeared in the work of Kazhdan \cite{art:kazh} in the context of uniform $\eps$-representations (of compact or amenable groups), a concept related to asymptotic representations, abeit essentially different.

Recall that the classical Kazhdan's Property (T) is equivalent to the statement that $H^1(\Gamma,\mathcal H_\pi)=0$ for all unitary representations $\pi\colon\Gamma\to\Uni(\mathcal H_\pi)$. 
We say that a group is $n$-Kazhdan if $H^n(\Gamma,\mathcal H_\pi)=\{0\}$ for every unitary representation
$\pi\colon\Gamma\to\Uni(\mathcal H_\pi).$ Theorem \ref{main2} simply says that every $2$-Kazhdan group is Frobenius stable.
Thus to prove Theorem \ref{main} it suffices to find $2$-Kazhdan groups which are not residually finite. Now, the seminal work of Garland \cite{MR0320180} (as was extended by Ballmann-\'Swi\polhk atkowski \cite{MR1465598} and others -- see Section \ref{hika} for details) shows that for every $2\leq r\in\NN$ and $p$ large enough, cocompact (arithmetic) lattices in simple $p$-adic Lie groups of rank $r$ are $n$-Kazhdan for every $1\leq n<r$.
In fact, a variant of this has been used to give examples of groups with Property (T), i.e.\ $1$-Kazhdan groups, which are not linear (and potentially also not residually finite) by using ``exotic'' buildings of rank $2$ (see \cite{MR3343347}). We want to prove the existence of non-residually finite $2$-Kazhdan groups, but there is a catch: as $n=2$, $r$ should be at least $3$, but a well-known result of Tits %\fxnote{reference?} 
asserts that for $r\geq 3$, there are no ``exotic'' buildings of dimension $r$ and the standard ones coming from $p$-adic Lie groups provide lattices which are all residually finite.
To work around this point, we imitate a result (and method of proof) of Deligne \cite{MR507760}. Deligne showed that some non-uniform lattices in simple Lie groups (e.g.\ $\mathrm{Sp}(2n,\ZZ)$) have finite central extensions which are not residually finite. Raghunathan \cite{MR735524} extended it also to some cocompact lattices in ${\rm Spin}(2,n)$.
These examples became famous when Toledo \cite{MR1249171} used them to provide examples of fundamental groups of algebraic varieties which are not residually finite. In the last section, we will explain how Deligne's method can be applied also to cocompact lattices in certain $p$-adic Lie groups. Along the way we use the solution to the congruence subgroup problem for these lattices which was provided by Rapinchuk \cite{MR1015345} and Tomanov \cite{MR1022796}. This way we will get finite central extensions of certain cocompact $p$-adic lattices which are themselves not residually finite anymore. Finally, an easy spectral sequence argument shows that a finite (central) extension of an $n$-Kazhdan group is also $n$-Kazhdan. Thus, the non-residually finite central extensions of the abovementioned lattices provide the non-Frobenius-approximated group promised in Theorem \ref{main}.

\vspace{0.1cm}

Along the way in Section \ref{stabapp}, we also provide examples of residually finite groups which are not Frobenius-stable and of finitely generated non-residually finite groups which are Frobenius-approximated. It is currently unclear if maybe all amenable (or even all solvable) groups are Frobenius-approximated. Moreover, it is an open problem to decide if the class of Frobenius-approximated groups is closed under central quotients or under crossed products by $\mathbb Z$, compare with \cite{MR3320894, MR2566306}.

\vspace{0.1cm}

The results of this article are part of the PhD project of the first named author.

\subsection{Notation}
Given any set $S$ we let $\mathbb F_S$ denote the free group on $S$. For any $R\del \mathbb F_S$ we let $\llangle R\rrangle$ denote the normal subgroup generated by $R$ and we let $\langle S\mid R\rangle:=\mathbb F_S/\llangle R\rrangle$ be the group with generators $S$ and relations $R$.
We use the convention $\NN=\{1,2,\ldots\}$.
For $n\in\NN$ we let $ {\rm M}_{n}(\mathbb C)$ denote the complex $n\times n$-matrices and ${\rm U}(n)\del {\rm M}_{n}(\mathbb C)$ the group of unitary matrices. The identity matrix is denoted by $1_n$.

Recall, an ultrafilter $\uf$ on $\NN$ is a non-trivial collection of subsets of $\NN$, such that (i) $A \in \uf$, $A \subset B$ implies $B \in \uf$, (ii) $A,B \in \uf$ implies $A\cap B \in \uf$, and (iii) $A \not \in \uf$ if and only if $\NN \setminus A \in \uf$ holds. We say that $\uf$ is non-principal if $\{n\} \not \in \uf$ for all $n \in \NN$. The existence of non-principal ultrafilters on $\NN$ is ensured by the Axiom of Choice. We can view a non-principal ultrafilter as a finitely additive probability measure defined on all subsets of $\NN$, taking only the values $\{0,1\}$ and giving the value $0$ to all finite subsets of $\NN$.

Throughout the whole paper, we fix a non-principal ultrafilter $\uf$ on $\NN$. Given some statement $P(n)$ for $n\in\NN$, we use the wording  \emph{$P(n)$ holds for most $n\in\NN$} as $\{n\in\NN\mid P(n)\}\in \uf$. Given a bounded sequence $(x_n)_{n\in\NN}$ of real numbers we denote the limit along the ultrafilter by $\lim_{n\to\uf}x_n\in(-\infty,\infty)$. Formally, the limit is the unique real number $x$ such for all $\varepsilon>0$ we have
$\{n \in \NN \mid |x_n - x|< \varepsilon \} \in \uf.$ For unbounded sequences, the limit takes a well-defined value in the extended real line $[-\infty,\infty]$.

We adopt the Landau notation; given two sequences $(x_n)_{n\in\NN}$ and $(y_n)_{n\in\NN}$ of non-negative real numbers, we write $x_n=O_\uf(y_n)$ if there exists $C>0$ such that $x_n\leq Cy_n$ for most $n\in\NN$ and $x_n=o_\uf(y_n)$ if there exists a third sequence $(\eps_n)_{n\in\NN}$ of non-negative real numbers such that $\lim_{n\to\uf}\eps_n=0$ and $x_n=\eps_ny_n$.

\subsection{Unitarily invariant norms}
Recall that a norm $\lVert\cdot\rVert$ on $ {\rm M}_d(\mathbb C)$ is called \emph{unitarily invariant} if
\[\lVert UAV\rVert=\lVert A\rVert,\]
for all $A\in {\rm M}_d(\mathbb C)$ and $U,V\in{\rm U}(d)$.

Important examples of such norms are the operator norm $\opnorm T=\sup_{\Vert v\Vert=1}\Vert Tv\Vert,$ the Frobenius norm $\frob T=\Tr(T^*T)^{1/2}=\sqrt{\sum_{i,j}^n|T_{ij}|^2}$ (also known as the unnormalized Hilbert-Schmidt norm) and the normalized Hilbert-Schmidt norm (or 2-norm) given by $\hsnorm T=\frac{1}{\sqrt n}\frob T$  for $T\in {\rm M}_{n}(\mathbb C)$. 
Here, $T^*$ denotes the adjoint matrix and $T$ is called self-adjoint if $T=T^*$. The matrix $T$ is called unitary if $TT^*=T^*T=1_n$.
We recall some basic and well-known facts about unitarily invariant norms. For $T \in {\rm M}_{n}(\mathbb C)$, we set $|T|=(T^*T)^{1/2}$. For self-adjoint matrices $A,B$, we write $A \leq B$ if $B-A$ is positive semi-definite, i.e. if $B-A$ has only non-negative eigenvalues.

\begin{prop}\label{uniprop}
Let $A,B,C\in {\rm M}_d(\mathbb C)$. Then, for any unitarily invariant norm, it holds that
\begin{enumerate}
\item[$(1)$] $\Vert ABC\Vert\leq \opnorm{A}\Vert B\Vert\opnorm{C},$
\item[$(2)$] $\lVert A\rVert=\lVert A^*\rVert=\lVert |A|\rVert$,
\item[$(3)$] If $A$ and $B$ are positive semi-definite matrices and $A\leq B$, then $\lVert A\rVert\leq \lVert B\rVert.$
\end{enumerate}
\end{prop}

\begin{comment}
\begin{proof}
For (1), we may assume that $\opnorm A=\opnorm C=1$.
Since the unit ball in $ {\rm M}_d(\mathbb C)$ is the convex hull of ${\rm U}(d)$, we can write $A=\sum_{i=1}^m\lambda_i U_i$ and $C=\sum_{j=1}^n\mu_jV_j$ for $U_i,V_j\in {\rm U}(d)$ and $\lambda_i,\mu_j\in[0,1]$ with $\sum_{i=1}^m\lambda=\sum_{j=1}^n\mu_j=1$.
Then, by unitary invariance,
\[\Vert ABC\Vert\leq\sum_{i=1}^m\sum_{j=1}^n\Vert U_iBV_j\Vert=\Vert B\Vert.\]

Now, for (2), we can write $A=U|A|$ for some unitary $U\in{\rm U}(d)$, whence it follows that $\lVert A\rVert=\lVert |A|\rVert$. Taking adjoints we reach the desired conclusion.

To prove (3), recall that if $0\leq A\leq B$, then there exists $D\in {\rm M}_d(\mathbb C)$ with $\opnorm D\leq 1$ and $A=DB$. Thus the result follows from (1).
\end{proof}
\end{comment}

\begin{prop}\label{quadclose}
Let $A\in{\rm U}(d)$. Then there is a unitary $B\in{\rm U}(d)$ such that $B^2=1$ and
\[\lVert B-A\lVert\leq \lVert 1_d-A^2\rVert,\]
for all unitarily invariant norms.
\end{prop}
\begin{proof}
By unitary invariance, we may assume that $A$ is a diagonal matrix, which we denote $A=\diag(a_1,\ldots,a_d)$. Let
\[b_j:=\begin{cases}
1,&\text{if }\Ra(a_j)\geq 0,\\
-1,&\text{if }\Ra(a_j)<0.
\end{cases}\]
One readily sees that $|b_j-a_j|\leq |1-a_j||-1-a_j|= |1-a_j^2|$.
Thus $B=\diag(b_1,\ldots,b_d)$ is a self-adjoint unitary and by Proposition \ref{uniprop} (3)
\[\Vert B-A\Vert= \Vert |B-A|\Vert \leq \Vert |1_d-A^2|\rVert=\Vert 1_d-A^2\Vert.\qedhere\]
\end{proof}

A second property that is important to us is \emph{submultiplicativity}, that is, $\lVert AB\rVert\leq \lVert A\rVert\lVert B\rVert$ for all $A,B\in {\rm M}_d(\mathbb C)$. This property turns $( {\rm M}_d(\mathbb C),\lVert\cdot\rVert)$ into a Banach algebra. The operator norm and the Frobenius norm enjoy this property, but the normalized Hilbert-Schmidt norm does not.

\subsection{Ultraproducts}\label{ultra} We will need the definition of the ultraproduct of Banach spaces and metric groups, respectively. 
First, let $(V_n)_{n\in\NN}$ be a sequence of Banach spaces. Consider the $\ell^\infty$-direct product $\prod_{n\in\NN} V_n$ (i.e. the Banach space of bounded sequences $(v_n)_{n\in\NN}$ with $v_n\in V_n$) and the closed subspace of nullsequences
\[I:=\left\{(v_n)_{n\in\NN}\in\prod_{n\in\NN}V_n\mid \lim_{n\to\uf}\Vert v_n\Vert_{V_n}=0\right\}.\]
We define the \emph{ultraproduct Banach space} by
\[\prod_{n\to\uf}(V_n,\lVert\cdot\rVert_{V_n}):=\prod_{n\in\NN} V_n\Big/ I.\]
As the name suggests, the ultraproduct Banach space is itself a Banach space with the norm induced by $\Vert (x_n)_{n\in\NN}\Vert=\lim_{n\to\uf}\Vert x_n\Vert_{V_n}$ for $(x_n)_{n\in\NN}\in\prod_{n\in\NN}V_n$. Moreover, if the $V_n$ are all Banach algebras, $C^*$-algebras or Hilbert spaces, so is the ultraproduct.

Let $(G_n)_{n\in\NN}$ be a family of groups, all equipped with bi-invariant metrics $d_n$.
In this case, the subgroup
\[N=\left\{(g_n)_{n\in\NN}\in\prod_{n\in\NN}G_n\mid \lim_{n\to\uf}d_n(g_n,1_{G_n})=0\right\}\]
of the direct product $\prod_{n\in\NN} G_n$
is normal, so we can define the \emph{metric ultraproduct}
\[\prod_{n\to\uf}(G_n,d_n):=\prod_{n\in\NN}G_n\Big/N.\]
Note that, in contrast to the Banach space definition \emph{we do not require the sequences to be bounded}. It is worth noting (albeit not relevant for our purposes) that the bi-invariant metric
\[d((g_n)_{n\in\NN},(h_n)_{n\in\NN})=\lim_{n\to\uf}\min\{d_n(g_n,h_n),1\},\qquad g_n,h_n\in G_n\]
on $\prod_{n\in\NN} G_n$ induces a bi-invariant metric on $\prod_{n\to\uf}(G_n,d_n)$.

The above definitions will be relevant to us in the following setting. Let $(k_n)_{n\in\NN}$ be a sequence of natural numbers and consider the family of matrix algebras $V_n:= {\rm M}_{k_n}(\mathbb C)$ equipped with some unitarily invariant, submultiplicative norms $\lVert\cdot\rVert_n$. We usually omit the index and denote all the norms by $\lVert\cdot\rVert$. Let $G_n:={\rm U}(k_n)$, equipped with the metrics $\dist_{\lVert\cdot\rVert_n}(g,h)=\lVert g-h\rVert_n,g,h\in G_n$ induced from the norms.
We consider the ultraproduct Banach space 
\[\Matuf:=\prod_{n\to\uf}( {\rm M}_{k_n}(\mathbb C),\lVert\cdot\rVert_n),\]
and the metric ultraproduct
\[\Uniuf:=\prod_{n\to\uf}({\rm U}(k_n),\dist_{\lVert\cdot\rVert_n}).\]
By submultiplicativity of the norms, we see that if $u_n\in{\rm U}(k_n)$ with $\lim_{n\to\uf}\dist_{\lVert\cdot\rVert_n}(u_n,1_{k_n})=0$ then
\[\lim_{n\to\uf}\Vert u_nT_n-T_n\Vert_n\leq \lim_{n\to\uf}\lVert u_n-1_{k_n}\rVert_n\cdot\lVert T_n\rVert_n=0,\]
for all bounded sequences $T_n\in {\rm M}_{k_n}(\mathbb C)$. Thus left multiplication by $u_n$ induces a left action of $\Uni_\uf^{\lVert\cdot\rVert}$ on ${\mathrm M}_{\uf}^{\lVert\cdot\rVert}$. By unitary invariance of the norms, we see that this action is isometric. Similarly, we have a right action by right multiplication and another left action by conjugation -- both of them isometric.

\subsection{Asymptotic homomorphisms} \label{apphom}
In this section, we let $\Gamma=\langle S\mid R\rangle$ be a fixed finitely presented group (i.e.\ $S$ and $R$ are finite) and we let $\gclass$ be a class of groups, all equipped with bi-invariant metrics. Any map $\ph\colon S\to G$, for some $G\in \gclass$, uniquely determines a homomorphism $\mathbb F_S\to G$ which we will also denote by $\ph$.

\begin{defi}
Let $G\in\gclass$ and let $\ph,\psi\colon S\to G$ be maps. The \emph{defect} of $\ph$ is defined by
\[\defect(\ph):=\max_{r\in R}d_G(\ph(r),1_G).\]
The \emph{distance} between $\ph$ and $\psi$ is defined by
\[\dist(\ph,\psi)=\max_{s\in S}d_G(\ph(s),\psi(s)).\]
The \emph{homomorphism distance} of $\ph$ is defined by
\[\homdist(\ph):=\inf_{\pi\in\Hom(\Gamma,G)}\dist(\ph,\pi|_S).\]
\end{defi}

\begin{defi}
A sequence of maps $\ph_n\colon S\to G_n$, for $G_n\in\gclass$, is called an \emph{asymptotic homomorphism} if $\lim_{n\to\uf}\defect(\ph_n)=0$.
\end{defi}
We will mainly be concerned with finite dimensional asympotic \emph{representations}, that is, asymptotic homomorphisms with respect to the class of unitary groups ${\rm U}(n)$ on finite dimensional Hilbert spaces, equipped with the metrics
\[d(T,S)=\Vert T-S\Vert,\qquad T,S\in{\rm U}(n),\]
coming from some family of unitarily invariant norms $\lVert\cdot\rVert$. 
The class of finite-dimensional unitary groups with metrics coming from$\lVert\cdot\rVert_{\mathrm{op}}$, $\lVert\cdot\rVert_{\mathrm{Frob}}$ and $\lVert\cdot\rVert_{\mathrm{HS}}$ are denoted $\Uni_{\mathrm{op}}$, $\Uni_{\mathrm{Frob}}$, and $\Uni_{\mathrm{HS}}$.

We might also find the need to quantify the above definition.
\begin{defi}
Let $\eps>0$ and $G\in\gclass$. An \emph{$\eps$-almost homomorphism} is a map $\ph\colon S\to G$ such that 
$\defect(\ph)\leq \eps.$
\end{defi}

In the literature, there are many different (inequivalent) notions of ``almost'', ``asymptotic'' and ``quasi-'' homomorphisms. If one would be precise, the above notion of asymptotic homomorphism could be called a \emph{local, discrete} asymptotic homomorphism. Local, since we are only interested in the behaviour of $\ph_n$ on the set of relations $R$ (compare with the \emph{uniform} situation \cite{MR3038548}) and discrete, because the family of homomorphisms are indexed by the natural numbers.

\begin{defi}
Let $G_n\in\gclass,n\in \NN$. Two sequences $\ph_n,\psi_n\colon S\to G_n$ are called \emph{(asymptotically) equivalent} if $\lim_{n\to\uf}\dist(\ph_n,\psi_n)=0.$
\end{defi}
If an asymptotic homomorphism $(\ph_n)_{n\in\NN}$ is equivalent to a sequence of genuine representations, we call $(\ph_n)_{n\in\NN}$ \emph{trivial} or \emph{liftable}.

We will now come to two central notions that we study in this paper, the notion of stability and approximability by a class of metric groups.

\begin{defi}\label{stabdef}
The group $\Gamma$ is called \emph{$\gclass$-stable} if all asymptotic homomorphisms are equivalent to a sequence of homomorphisms, that is,
\[\lim_{n\to\uf}\homdist(\ph_n)=0,\]
for all $\ph_n\colon S\to G_n$, $G_n\in\gclass,n\in\NN$ with $\lim_{n\to\uf}\defect(\ph_n)=0$.
\end{defi}

\begin{defi}
A finitely presented group $\Gamma=\langle S\mid R \rangle$ is called \emph{$\gclass$-approximated}, if there exists an asymptotic homomorphism $\ph_n\colon S \to G_n$, $G_n\in\gclass,n\in\NN$ such that 
\[\lim_{n\to\uf}d_n(\ph_n(x),1_{G_n})>0,\qquad\text{for all\ } x\in  \mathbb F_S\backslash \llangle R\rrangle.\]
\end{defi}
We will be mainly concerned with $\Uni_{\rm Frob}$-approximation and $\Uni_{\rm Frob}$-stability in this paper and,  for convenience, we will often just speak about Frobenius-approximation and Frobenius-stability in this context.

\begin{defi}
A group $\Gamma$ is called \emph{residually $\gclass$} if for all $x\in\Gamma\backslash\{1_\Gamma\}$ there is a homomorphism $\pi\colon \Gamma\to G$ for some $G\in\gclass$ such that $\pi(x)\neq 1_G$.
\end{defi}

The following proposition (see \cite{glebskyrivera2008sofic} or \cite{MR3350728}) is evident from the definitions, nevertheless a central observation in our work. 

\begin{prop} Let $\Gamma$ be a finitely presented. If $\Gamma$ is $\gclass$-stable and $\gclass$-approximated group, then it must be residually $\gclass$. In particular, if the class $\gclass$ consists of finite-dimensional unitary groups, any finitely presented, $\gclass$-stable and $\gclass$-approximated group is residually finite.
\end{prop}

We finish this section with a basic lemma. The important part in the statement of the lemma is that $K_r$ does not depend on $\ph$.
\begin{lemma}\label{relfact}
For all $r\in\llangle R\rrangle$ there is a constant $K_r$ such that for all groups $G$ with a bi-invariant metric and all maps $\ph\colon S\to G$ it holds that
\[d_{G}(\ph(r),1_{G}) \leq K_r  \defect(\ph).\]
\end{lemma}
\begin{proof}
If $r\in\llangle R\rrangle$ we can determine $r_1,\ldots,r_k\in R\cup R^{-1}$ and $x_1,\ldots,x_k\in\mathbb F_S$ such that
\[r=x_1r_1x_1^{-1} x_2r_2x^{-1}_2\cdots x_kr_kx_k^{-1}.\]
Note that by bi-invariance \[d_{G}(\ph(r_j),1_{G})=d_{G}(\ph(r^{-1}_j),1_{G})\leq \defect(\ph)\] for all $j$.
Thus, using bi-invariance again, we get
\begin{align*}
d_{G}(\ph(r),1_{k_n})&=d_{G}(\ph(x_1r_1x_1^{-1})\cdots \ph(x_kr_kx^{-1}_k),1_{G})
\\&\leq \sum_{j=1}^kd_{G}(\ph(x_j)\ph(r_j)\ph(x_j)^{-1},1_{G})
\\&=\sum_{j=1}^k d_{G}(\ph(r_j),1_{G})
\\&\leq k\cdot \defect(\ph),
\end{align*}
So letting $K_r=k$, we are done.
\end{proof}

\begin{comment}

\begin{cor}
Let $\ph_n\colon S\to G_n$ be an asymptotic homomorphism. Then the map
$\ph_\uf\colon\Gamma\to\Uniuf$
given by $\ph_\uf(g)=(\ph_n(g))_{n\to\uf}$ is a homomorphism. (Here $(\ph_n(g))_{n\to\uf}$ denotes the element in $\Uniuf$
represented by $(\ph_n(g))_{n\in\NN}\in \prod_{n\in\NN}\Uni(k_n)$.)
\end{cor}
\begin{rem}
We note that if $\ph_n$ is an $\eps_n$-almost representation and $\lim_{n\to\uf}\eps_n=0$, then Lemma \ref{relfact} shows that $(\ph_n(r))_{n\in\NN}\in N(G_n,d_{G_n},O_\uf(\eps_n))$,
so we even get a homomorphism
\[\hat\ph_\uf\colon \Gamma\to\prod_{n\to\uf}(G_n,d_{G_n},O_\uf(\eps_n))\]
\end{rem}

\end{comment}

\subsection{Group cohomology}\label{sec:cohom}
For convenience, we recall one construction of group cohomology. We primarily need the second cohomology of a group with coefficients in a unitary representation, but for completeness, we give a more general definition. Let $\Gamma$ be any group and let $V$ be a $\Gamma$-module, i.e. an abelian group together with a (left) action $\pi$ of $\Gamma$ on $V$. We consider the chain complex $C^n(\Gamma,V),n\geq 1$, which is the set of functions from $\Gamma^n$ to $V$ together with the coboundary operator,
\[d=d^{n}\colon C^{n}(\Gamma,V)\to C^{n+1}(\Gamma,V),\]
defined by
\begin{align*}
d^{n}(f)(g_1,\ldots,g_{n+1})&=\pi(g_1)f(g_2,\ldots,g_{n+1})\\&+\sum_{j=1}^{n} (-1)^jf(g_1,\ldots,g_jg_{j+1},\ldots,g_{n+1})\\&+(-1)^{n+1}f(g_1,\ldots,g_{n}).
\end{align*}
We also let $C^0(\Gamma,V)=V$ and $d^0(v)(g)=\pi(g)v-v$ for $v\in V,g\in\Gamma$.
Thus, for $n\geq 0$ we define the \emph{$n$-coboundaries} to be $B^n(\Gamma,V)=\mathrm{Im}(d^{n-1})$ (with $B^0(\Gamma,V)=\{0\}$) and the \emph{$n$-cocycles} to be $Z^n(\Gamma,V)=\ker(d^{n})$. One checks that $B^n(\Gamma,V)\del Z^n(\Gamma,V)$, so we can define the \emph{$n$-th cohomology} to be
\[H^n(\Gamma,V)=Z^n(\Gamma,V)/B^n(\Gamma,V).\]

Recall that given an extension of groups
\[1\to V\stackrel{i}{\to} \hat\Gamma\stackrel{q}{\to} \Gamma\to 1,\]
where $V$ is abelian, there is an action of $\Gamma$ on $V$ induced by the conjugation action of $\hat\Gamma$ on $i(V)$. Fixing any section $\sigma\colon \Gamma\to\hat\Gamma$ (with $\sigma(1_\Gamma)=1_{\hat\Gamma}$) of the quotient $q$ we can define a map map $f\colon \Gamma\times \Gamma\to V$ as the solution to \[i(f(g,h))=\sigma(g)\sigma(h)\sigma(gh)^{-1},\] for $g,h\in\Gamma$.
It is straightforward to check that $f\in Z^2(\Gamma,V,\pi)$ and $f\in B^2(\Gamma,V,\pi)$ exactly when the extension \emph{splits}, i.e.\ there is a homomorphism $p\colon \Gamma\to\hat\Gamma$ such that $q\circ p=\id_\Gamma$.

Assume now that $\Gamma$ is countable and $V$ is a Banach space with norm $\lVert\cdot\rVert$, then we can define a separating family of semi-norms on $C^n(\Gamma,V)$ by
\begin{equation} \label{eqdefnorm}
\Vert f\Vert_F=\max_{g\in F}\Vert f(g)\Vert,\end{equation}
for $f\in C^n(\Gamma,V)$ and finite $F\del \Gamma^n$. It is easy to see that with respect to this family, $C^n(\Gamma,V)$ is a Fréchet space (one can even take $\Vert \cdot\Vert_{\{x\}},x\in \Gamma^n$ as separating family) and if $\Gamma$ acts on $V$ by isometries, the map $d^{n}$ is bounded.

\section{Some examples of non Frobenius-stable groups} \label{stabapp}

Part of our aim is to provide a large class of Frobenius-stable groups, but let us start out by giving examples of well-known groups that are \emph{not} stable.
Specifically, we show that $\Z^2$ and the Baumslag-Solitar group $\mathrm{BS}(2,3)$  are not Frobenius-stable by giving concrete examples of asymptotic representations that are not equivalent to genuine representations. We also exploit the latter example to provide an example of an Frobenius-approximated, non-residually finite group, see Section \ref{grexample}.

\subsection{$\Z^2$ is not Frobenius-stable}
In \cite{V1}, Voiculescu proved that the matrices
\[A_n=\begin{pmatrix}
1&&&&\\
&\omega_n&&&\\
&&\omega_n^2&&\\
&&&\ddots&\\
&&&&\omega_n^{n-1}
\end{pmatrix},\qquad B_n=\begin{pmatrix}
0&& &0&1\\
1&0&&&0\\
&1&&&\\
&&\ddots&&\\
&&&1&0
\end{pmatrix}\in {\rm U}(n),\]
where $\omega_n:=\exp(\tfrac{2\pi i}{n}),n\in\NN$, define a non-trivial $\opnorm\cdot$-asymptotic representation of $\Z^2=\ls{a,b,\mid aba^{-1}b^{-1}}$ by $\ph_n(a)=A_n$ and $\ph_n(b)=B_n$. More precisely,
\[\defect_{\opnorm\cdot}(\ph_n)=\opnorm{A_nB_nA_n^*B_n^*-1_n}=|\omega_n-1|=O_\uf\big(\tfrac{1}{n}\big),\]
but
\[\homdist_{\opnorm\cdot}(\ph_n)\geq \sqrt{2-|1-\omega_n|}-1\]
(see also \cite{exelloring, V1}).
By the inequalities $\opnorm T\leq \frob T\leq n^{1/2}\opnorm T$ for $T\in{\rm U}(n)$, we conclude that
\[\defect_{\frob\cdot}(\ph_n)=O_\uf(n^{-1/2}),\]
and
\[\homdist_{\frob\cdot}(\ph_n)\geq \sqrt{2-|1-\omega_n|}-1,\]
so $\ph_n$ is also a non-trivial $\frob\cdot$-asymptotic representation. In particular $\ZZ^2$ is neither $\Uni_{\mathrm{op}}$- nor Frobenius-stable. It is worth noting that $\ZZ^2$ actually \emph{is} $\Uni_{\mathrm{HS}}$-stable, see e.g.\ \cite{glebcom} for a quantitative proof.
%\fxnote{I couldn't seem to find the reference published anywhere. Maybe there is another reference?}).

\subsection{$\mathrm{BS}(2,3)$ is not Frobenius-stable}
We now turn our attention to the Baumslag-Solitar group $\mathrm{BS}(2,3)=\ls{a,b\mid b^{-1}a^2ba^{-3}}$, see \cite{MR0142635} for the original reference.
By definition, the generators satisfy the equation
\begin{equation}\label{eq_BS}
b^{-1}a^2b=a^3.
\end{equation}
It is also well known and not hard to check that the generators do \emph{not} satisfy
\begin{equation}\label{eq_comm}
ab^{-1}ab=b^{-1}aba.
\end{equation} 
Indeed, this follows easily from the description of ${\rm BS}(2,3)$ as an HNN-extension of $\mathbb Z$.
On the other hand, we recall the following.
\begin{prop}[Baumslag-Solitar \cite{MR0142635}]\label{prop1}
Let $\Gamma$ be a residually finite group. If $a,b\in \Gamma$ satisfy \eqref{eq_BS}, then they also satisfy \eqref{eq_comm}.
\end{prop}
\begin{proof}
Indeed, if $a$ has finite order and $a^2$ is conjugate to $a^3$, then the order of $a$ cannot be even. Thus, $b^{-1}ab$ is a power of $b^{-1}a^2b=a^3$. We conclude that $a$ and $b^{-1}ab$ commute.
\end{proof}

%\begin{cor}\label{cor1}
%Let $\Gamma$ be a residually finite group  and $a,b\in \Gamma$ satisfy Equation \eqref{eq_BS}. Then $a,b$ satisfy Equation \eqref{eq_comm}.
%\end{cor}
By Mal'cev's Theorem we immediately obtain the following consequence.
\begin{cor}\label{cor2}
Let $a,b$ be unitary matrices. If $a,b\in \Gamma$ satisfy \eqref{eq_BS}, then they also satisfy \eqref{eq_comm}.
\end{cor}

This last Corollary can also be proven directly by linear algebra methods, see \cite{fritz} where some quantitative aspects of operator-norm aproximability of $\mathrm{BS}(2,3)$ were studied.
By Corollary~\ref{cor2}, in order to show that $\mathrm{BS}(2,3)$ is non-stable it suffices to find a sequence of pairs of unitary matrices that $\frob\cdot$-asymptotically satisfy Equation \eqref{eq_BS} but are far from satisfying
Equation \eqref{eq_comm}. The study of approximation properties of $\mathrm{BS}(2,3)$ goes back to R\u adulescu \cite{MR2436761}, where the focus was more on approximation in the (normalized) Hilbert-Schmidt norm. We are now going to prove the following result.

\begin{theorem}\label{th_bs_approx}
The group ${\rm BS}(2,3)$ is not Frobenius-stable.
\end{theorem}

The theorem is a direct consequence of the following lemma.
\begin{lemma}
There exist $A_n,B_n\in {\rm U}(6n)$ such that
\begin{itemize}
\item $\frob{B_n^{-1}A_n^2B_n-A_n^3}=O_\uf(\tfrac{1}{n})$
\item $\frob{A_nB_n^{-1}A_nB_n-B_n^{-1}A_nB_nA_n}= \sqrt{6n}-O_\uf(1)$.
\end{itemize}
\end{lemma}    
\begin{proof}
We will omit the index and write $A=A_n$ and $B=B_n$.
Let $\omega=\exp(\tfrac{2\pi}{6n})$ and consider a $6n$-dimensional Hilbert space $\cH$ with orthonormal basis $v[0],v[1],\dots v[6n-1]$.
Define $A\in{\rm U}(6n)$ as $Av[j]=\omega^jv[j]$ (that is, $A$ is $A_{6n}$ from the previous example). We plan to decompose $\cH$ as a direct sum in two ways $\cH=\oplus_{j=0}^{n-1} \cS[j]$ and
 $\cH=\oplus_{j=0}^{n-1} \cC[j]$ such that each $\cS[j]$ and each $\cC[j]$ is $6$-dimensional and the restriction of $A^2$ to $\cS[j]$ as well
as restriction of $A^3$ to $\cC[j]$ act approximately as multiplication by $\omega^{6j}$. 
(The letter $\cS$ stands for square and $\cC$ stands for cube.)   
Then we construct
$B=\oplus_{j=0}^{n-1} B_j$ with $B_j:\cC[j]\to\cS[j]$. Let us start the detailed construction.
Define
\begin{align*}
\cS[j]=\spam\{&v[3j],v[3j+1],v[3j+2],
\\&v[3j+3n],
v[3j+3n+1],v[3j+3n+2]\}
\end{align*}
and
\begin{align*}
\cC[j]=\spam\{&v[2j],v[2j+2n],v[2j+4n],
\\&v[2j+1],v[2j+2n+1],v[2j+4n+1]\}.
\end{align*}
We will use the ordered base of $\cS[j]$ (resp. $\cC[j]$) as it appears in their definitions.
Let $S_j$ (resp. $C_j$) be a restriction of $A$ to $\cS[j]$ (resp. $\cC[j]$).
Observe that \[S_j=\omega^{3j}\diag(1,\omega,\omega^2,-1,-\omega,-\omega^2)\] and
\[C_j=\omega^{2j}\diag(1,\exp(\tfrac{2\pi i}{3}),\exp(\tfrac{4\pi i}{3}),\omega,\omega\exp(\tfrac{2\pi i}{3}),\omega\exp(\tfrac{4\pi i}{3})).\]

Now, let $B\in{\rm U}(6n)$ be any unitary of the form
$B=\oplus_{j=0}^{n-1} B_j$ with unitary $B_j\colon\cC[j]\to\cS[j]$. We claim that
\[\frob{B^{-1}A^2B-A^3}=O_\uf(\tfrac{1}{n})\]

Indeed,
\[\frob{B^{-1}A^2B-A^3}^2=\sum_{j=0}^{n-1}\frob{B_j^{-1}S_j^2B_j-C_j^3}^2,\] 
and we obtain
\begin{align*}
\frob{B_j^{-1}S_j^2B_j-C_j^3}^2&=\frob{B_j^{-1}(S_j^2-\omega^{6j})B_j-(C_j^3-\omega^{6j})}^2
\\&\leq
\frob{S_j^2-\omega^{6j}}^2+\frob{C_j^3-\omega^{6j}}^2
\\&
=2(|1-\omega^2|^2+|1-\omega^4|^2)+3|1-\omega^3|^2=O_\uf(\tfrac{1}{n^2}),
\end{align*}
which entails the claim.
Now, consider the unitary given by the matrix
\[
B_j=\frac{1}{\sqrt{2}}\begin{pmatrix}
1 & 1 & 0 & 0 & 0 & 0 \\
1 &-1 & 0 & 0 & 0 & 0 \\
0 & 0 & 1 & 1 & 0 & 0 \\
0 & 0 & 1 &-1 & 0 & 0 \\
0 & 0 & 0 & 0 & 1 & 1\\
0 & 0 & 0 & 0 & 1 &-1 
\end{pmatrix},
\]
and let \[\tilde S=\diag(1,1,1,-1,-1,-1)\] and 
\[\tilde C =\diag(1,\exp(\tfrac{2\pi i}{3}),\exp(\tfrac{4\pi i}{3}),1,\exp(\tfrac{2\pi i}{3}),\exp(\tfrac{4\pi i}{3})).\]
It is not hard to check that \[\frob{S_j-\omega^{3j}\tilde S}^2=O_\uf(\tfrac{1}{n^2})\] and 
\[\frob{C_j-\omega^{2j}\tilde C}^2=O_\uf(\tfrac{1}{n^2}).\]  Direct calculations show that
$\frob{\tilde C B_j^{-1}\tilde SB_j-B_j^{-1}\tilde SB_j\tilde C}^2=6$, so since
\[\frob{AB^{-1}AB-B^{-1}ABA}^2=\sum\limits_{j=0}^{n-1}\frob{C_jB_j^{-1}S_jB_j-B_j^{-1}S_jB_jC_j}^2,\]
the lemma follows.
\end{proof}

\subsection{An example of a finitely generated, non-residually finite, and Frobenius-approximated group}
\label{grexample}
Note that the example above provides a homomorphism into the ultraproduct $\ph\colon \mathrm{BS}(2,3)\to\Uni_\uf^{\frob\cdot}$. The image $\Gamma=\ph(\mathrm{BS}(2,3))$ is clearly Frobenius-approximated, but it is clearly not residually finite, since, by construction, the elements $\ph(a),\ph(b)\in\Gamma$ satisfy \eqref{eq_BS} but not \eqref{eq_comm}. In some sense it is an artefact of the definitions that every non-Frobenius-stable group has a non-trivial Frobenius-approximated group quotient. It seems quite likely that the construction above is enough to show that $\mathrm{BS}(2,3)$ is itself Frobenius-approximated. Indeed, even though the proof of this assertion is not spelled out in full detail in \cite{MR2436761}, it appears that R\u adulescu's construction shows this. Note that it follows from work of Kropholler \cite{MR1078097} that ${\rm BS}(2,3)$ is residually solvable and hence MF, see \cite{cde}.

\section[Diminishing the defect]{Diminishing the defect of asymptotic representations}\label{sec:def}

This section contains the key technical novelty of this article. We associate an element $[\alpha]\in H^2(\Gamma,\prod_{n\to\uf}( {\rm M}_{k_n}(\mathbb C),\lVert\cdot\rVert))$ to an asymptotic representation $\ph_n\colon \Gamma\to{\rm U}(k_n)$. We prove that if $[\alpha]=0$, then the defect can be diminished in the sense that there is an equivalent asymptotic representation $\ph_n'$ with effectively better defect, more precisely $\defect(\ph_n')=o_\uf(\defect(\ph_n))$.

\subsection{Assumptions for this section}\label{maryassumption}
For this section, we fix the following.
\begin{itemize}
\item
A finitely presented group
$\Gamma=\langle S\mid R\rangle$,
\item a sequence of natural numbers $(k_n)_{n\in\NN}$,
\item a family of submultiplicative, unitarily invariant norms on ${\rm U}(k),k\in\NN$, all denoted by $\lVert\cdot \rVert$, and
\item an asymptotic representation $\ph_n:S\to{\rm U}(k_n)$ with respect to the metrics associated to $\lVert\cdot\rVert$.
\end{itemize}
Recall the ultraproduct notation introduced in Section \ref{ultra}, that is, $\Uniuf=\prod_{n\to\uf}({\rm U}(k_n),\dist_{\lVert\cdot\rVert})$ and $\Matuf=\prod_{n\to\uf}( {\rm M}_{k_n}(\mathbb C),\Vert \cdot\Vert)$ and recall that since $\lVert\cdot\rVert$ is submultiplicative, $\Uniuf$ acts on $\Matuf$ by multiplication.
An asymptotic representation as above induces a homomorphism $\ph_\uf:\Gamma\to \Uniuf$ on the level of the group $\Gamma$. Thus $\Gamma$ acts on $\Matuf$ through $\ph_\uf$.
With this in mind, we also want to fix the following:
\begin{itemize}
\item a section $\sigma\colon \Gamma\to\mathbb F_S$ of the natural surjection $\mathbb F_S \to \Gamma$, in particular, we have $\sigma(g)\sigma(h)\sigma(gh)^{-1}\in\llangle R\rrangle$ for all $g,h\in \Gamma$,
\item a sequence $\tilde\ph_n\colon \Gamma\to{\rm U}(k_n)$ such that $\tilde\ph(1_\Gamma)=1_{k_n}$, $\tilde\ph_n(g^{-1})=\tilde\ph_n(g)^*$, and for every $g\in \Gamma$
\begin{equation}\lVert \ph_n(\sigma(g))-\tilde\ph_n(g)\Vert=O_\uf(\defect(\ph_n)).\label{kinderriegel}
\end{equation}
In particular, the sequence $\tilde\ph_n$ is a lift of $\ph_\uf$.
\end{itemize}
For this, note that given any section $\sigma\colon \Gamma\to\mathbb F_S$, the sequence $\ph_n\circ\sigma$ is a lift of $\ph_\uf$. There exists a section $\sigma$ with $\sigma(1_\Gamma)=1_{\mathbb F_S}$ and $\sigma(g^{-1})=\sigma(g)^{-1}$ for all $g$ such that $g^2\neq 1_\Gamma$. We define $\tilde\ph_n(g):=\ph_n(\sigma(g))$ for all $g$ with $g^2\neq 1_\Gamma$. In the case $g^2=1_\Gamma$, by Lemma \ref{relfact} it holds that
\[\lVert\ph_n(\sigma(g))^2-1_{k_n}\rVert= O_\uf(\defect(\ph_n)),\] so by Proposition \ref{quadclose} there are self-adjoint unitaries $B_n\in{\rm U}(k_n)$ such that
\[\lVert B_n-\ph_n(\sigma(g))\rVert=O_\uf(\defect(\ph_n)).\]
By letting $\tilde\ph_n(g):=B_n$, we get the desired map.

\subsection{The cohomology class of an asymptotic representation}
We want to define an element in $H^2(\Gamma,\Matuf)$ associated to $\ph_n$.
To this end we define $c_n:=c_n(\ph_n):\Gamma\times \Gamma\to {\rm M}_{k_n}(\mathbb C)$ by
\[c_n(g,h)=
\frac{\tilde\ph_n(g)\tilde\ph_n(h)-\tilde\ph_n(gh)}{\defect(\ph_n)},\]
for all $n\in\NN$ such that $\defect(\ph_n)>0$ and $c_n(g,h)=0$ otherwise, for all $g,h\in \Gamma$.
The next proposition is a collection of basic properties of the maps $c_n$.
\begin{prop}\label{cprops}
Let $g,h,k\in\Gamma$. The maps $c_n$ satisfy the following equations
\begin{gather*}
\tilde\ph_n(g)c_n(h,k)-c_n(gh,k)+c_n(g,hk)-c_n(g,h)\tilde\ph_n(k)=0,\label{cocyc}
\\
c_n(g,g^{-1})=c_n(1_\Gamma,g)=c_n(g,1_\Gamma)=0\qquad \text{and}\qquad c_n(g,h)^*=c_n(h^{-1},g^{-1}),
\end{gather*}
Furthermore, we have for every $g,h \in \Gamma$
\begin{equation}
\Vert c_n(g,h)\Vert=O_\uf(1).\label{cbound}
\end{equation}

\end{prop}
\begin{proof}
For all $g,h,k\in \Gamma$ and $n\in\NN$ we have
\begin{align*}
&\defect(\ph_n)\cdot(\tilde\ph_n(g)c_n(h,k)-c_n(gh, k) + c_n(g, hk)-c_n(g, h)\tilde\ph_n(k))\\
&= \tilde\ph_n(g)(\tilde\ph_n(h)\tilde\ph_n(k)-\tilde\ph_n(hk))-(\tilde\ph_n(gh)\tilde\ph_n(k)-\tilde\ph_n(ghk))\\
&+ (\tilde\ph_n(g)\tilde\ph_n(hk)-\tilde\ph_n(ghk))-(\tilde\ph_n(g)\tilde\ph_n(h)-\tilde\ph_n(gh))\tilde\ph_n(k)
\\&= \tilde\ph_n(g)\tilde\ph_n(h)\tilde\ph_n(k)-\tilde\ph_n(g)\tilde\ph_n(hk)-\tilde\ph_n(gh)\tilde\ph_n(k)+\tilde\ph_n(ghk)\\&
+\tilde\ph_n(g)\tilde\ph_n(hk)-\tilde\ph_n(ghk)-\tilde\ph_n(g)\tilde\ph_n(h)\tilde\ph_n(k)+\tilde\ph_n(gh)\tilde\ph_n(k)\\
&= 0,
\end{align*}
which proves the first equation. The second line of equations is immediate from the definition of $c_n$ and the fact that $\tilde\ph_n(g^{-1})=\tilde\ph_n(g)^*$.

For the last assertion, note that since $\sigma(g)\sigma(h)\sigma(gh)^{-1}\in\llangle R\rrangle$,
it follows from Lemma \ref{relfact} that
\[\Vert \ph_n(\sigma(g)\sigma(h)\sigma(gh)^{-1})-1_{k_n}\Vert=O_\uf(\defect(\ph_n)).
\]
and thus it follows (by using Equation \eqref{kinderriegel}) that
\[\defect(\ph_n)\Vert c_n(g,h)\Vert=O_\uf(\defect(\ph_n)).\qedhere\]
\end{proof}

By \eqref{cbound} it follows that for every $g,h \in \Gamma$, $c_n(g,h)$ is a bounded sequence, so the sequence defines a map
\[c=(c_n)_{n\in\NN}\colon \Gamma\times \Gamma\to\Matuf.\]
This map is not a cocycle in the sense explained in Section \ref{sec:cohom}, but, as the next corollary states, the map $\alpha(g,h):=c(g,h)\ph_\uf(gh)^*$ is. (The map $c$ is a cocycle in the equivalent picture of \emph{Hochschild} cohomology and it turns out that some calculations are more natural with $c$, so we will also work with this map.)
Even though we suppress it in the notation, keep in mind that $c$ and $\alpha$ depend on the lift $\tilde\ph_n$ and on $\defect(\ph_n)$.

\begin{cor}
The map $\alpha \colon \Gamma\times \Gamma\to\Matuf$ is a 2-cocycle with respect to the isometric action $\pi(g)T=\ph_\uf(g)T\ph_\uf(g)^*,g\in\Gamma,T\in\Matuf$.
\end{cor}
\begin{proof}
Given $g,h,k\in \Gamma$ we have that
\begin{align*}
&\ph_\uf(g)\alpha(h,k)\ph_\uf(g)^{*}-\alpha(gh,k)+\alpha(g,hk)-\alpha(g,h)\\
&=\ph_\uf(g)c(h,k)\ph_\uf(hk)^{*}\ph_\uf(g)^{*}-c(gh,k)\ph_\uf(ghk)^{*}\\&\ +c(g,hk)\ph_\uf(ghk)^{*}-c(g,h)\ph_\uf(gh)^{*}\\
&=(\ph_\uf(g)c(h,k)-c(gh,k)+c(g,hk)-c(g,h)\ph_\uf(k))\ph_\uf(ghk)^{*}\\&=0,
\end{align*}
where we used that $\ph_\uf$ is a homomorphism and Proposition \ref{cprops}.
\end{proof}

We call $\alpha$ the cocycle associated to the sequence $(\varphi_n)_{n \in \mathbb N}$.

\begin{prop}\label{iota}
Assume that $\alpha$ represents the trivial cohomology class in $H^2(\Gamma,\Matuf)$, i.e.\ there exists a map $\beta\colon\Gamma\to\Matuf$ satisfying
\[\alpha(g,h)=\ph_\uf(g)\beta(h)\ph_\uf(g)^*-\beta(gh)+\beta(g),\qquad g,h\in\Gamma.\]
Then
\begin{gather}
\beta(1_\Gamma)=0,\label{en}\\
\beta(g)=-\ph_\uf(g)\beta(g^{-1})\ph_\uf(g)^*\label{to}\\
c(g,h)=\ph_\uf(g)\beta(h)\ph_\uf(h)-\beta(gh)\ph_\uf(gh)+\beta(g)\ph_\uf(gh).\label{tre}
\end{gather}
Furthermore, we can choose
$\beta(g)$ to be skew-symmetric for all $g\in\Gamma$.
\end{prop}
\begin{proof}
Equation \eqref{tre} is immediate from
$
c(g,h)=\alpha(g,h)\ph_\uf(gh)
$
for $g,h\in G$.
Equation \eqref{en} follows from \eqref{tre} and Proposition \ref{cprops} with $g=h=1_\Gamma$ and \eqref{to} follows from \eqref{en}, \eqref{tre} and Proposition \ref{cprops} with $h=g^{-1}$.
For the last claim, we possibly need to alter $\beta$ a little. Note that $\beta'(g):=-\beta(g)^*=\ph_\uf(g)\beta(g^{-1})^*\ph_\uf(g)^*$ also satisfies \eqref{en}-\eqref{tre}. Indeed,
\begin{align*}
&c(g,h)=c(h^{-1},g^{-1})^*\\
&=(\ph_\uf(h^{-1})\beta(g^{-1})\ph_\uf(g^{-1})-\beta(h^{-1}g^{-1})\ph_\uf(h^{-1}g^{-1})+\beta(h^{-1})\ph_\uf(h^{-1}g^{-1}))^*
\\&=\ph_\uf(g)\beta(g^{-1})^*\ph_\uf(h)-\ph_\uf(gh)\beta((gh)^{-1})^*+\ph_\uf(gh)\beta(h^{-1})^*
\\&=\beta'(g)\ph_\uf(gh)-\beta'(gh)\ph_\uf(gh)+\ph_\uf(g)\beta'(h)\ph_\uf(h)
\end{align*}
for $g,h\in \Gamma$, which proves \eqref{tre} whence the other two follow.
Thus, replacing $\beta$ with 
\[\beta^\sharp(g):=
\frac{\beta(g)-\beta(g)^*}{2},\qquad g\in \Gamma,n\in\NN
,\]
we see that $\beta^\sharp(g)$ is skew-symmetric and that Equations \eqref{en}-\eqref{tre} are still satisfied.
\end{proof}

\subsection{Correction of the asymptotic representation}
Now let $\beta$ be as above and let $\beta_n\colon \Gamma\to {\rm M}_{k_n}(\mathbb C)$ be any skew-symmetric lift of $\beta$.
Then $\exp(-\defect(\ph_n)\beta_n(g))$ is a unitary for every $g\in \Gamma$, so we can define a sequence of maps $\psi_n:\Gamma\to{\rm U}(k_n)$ by 
\[\psi_n(g)=\exp(-\defect(\ph_n)\beta_n(g))\tilde\ph_n(g).\]
Note that since $\tilde\ph_n(1_\Gamma)=1_{k_n}$ and $\beta_n(1_\Gamma)=0$, we have $\psi_n(1_\Gamma)=1_{k_n}$.

In the proofs of Proposition \ref{closeto} and Lemma \ref{lemspeed}, we will make use of two basic inequalities that hold for any $k \in \NN$ and $A \in {\rm M}_k(\mathbb C)$:
\begin{equation} \label{eqb1}
\|1_k-\exp(A)\| \leq \|A\| \exp(\|A\|)
\end{equation}
\begin{equation} \label{eqb2}
\|1_k-A-\exp(A)\| \leq \|A\|^2 \exp(\|A\|),
\end{equation}
They are simple consequences of the definition $\exp(A)=\sum_{k=0}^{\infty} \frac{A^k}{k!}$ and the triangle inequality and submultiplicativity of the norm.

\begin{prop}\label{closeto}
With the notation from above, for every $g\in\Gamma$, we have
\[\Vert \tilde\ph_n(g)-\psi_n(g)\Vert=O_\uf(\defect(\ph_n)).\]
More precisely,
\[\Vert\tilde\ph_n(g)-\psi_n(g)\Vert\leq 2\Vert \beta_n(g)\Vert\defect(\ph_n)\]
for most $n\in\NN$.
\end{prop}
\begin{proof}
Let $g\in\Gamma$.
By unitary invariance and submultiplicativity, we get that
\begin{align*}
\Vert\tilde\ph_n(g)-\psi_n(g)\Vert
&=\Vert 1_{k_n}-\exp(-\defect(\ph_n)\beta_n(g))\Vert
\\&\stackrel{\eqref{eqb1}}{\leq} \defect(\ph_n)\Vert\beta_n(g)\Vert\exp(\defect(\ph_n)\Vert\beta_n(g)\Vert)
\end{align*}
and since $\Vert\beta_n(g)\Vert$ is a bounded sequence and $\lim_{n\to\uf}\defect(\ph_n)=0$, we have $\exp(\defect(\ph_n)\Vert\beta_n(g)\Vert)\leq 2$ for most $n$ and the result follows.
\end{proof}

It follows that $\psi_n|_S$ is an asymptotic representation with $\defect(\psi_n|_S)=O_\uf(\defect(\ph_n))$, but we prove that the defect is actually $o_\uf(\defect(\ph_n))$.

\begin{lemma}\label{lemspeed}
For any $g,h\in \Gamma$, we have that
\[\Vert \psi_n(gh)-\psi_n(g)\psi_n(h)\Vert=o_\uf(\defect(\ph_n)).\]
\end{lemma}
\begin{proof}
Let $\xi_n(x):=(1_{k_n}-\defect(\ph_n)\beta_n(x))\tilde\ph_n(x)$, for $x\in\Gamma$, and
let $g,h\in \Gamma$ be fixed. Let $C=2\max_{x\in \{g,h,gh\}}\Vert \beta(x)\Vert$.
Whence it follows that for most $n\in\NN$,
\begin{align*}
\Vert\psi_n(x)-\xi_n(x)\Vert \stackrel{\eqref{eqb2}}{\leq} C\cdot\defect(\ph_n)^2,
\end{align*}
for $x\in \{g,h,gh\}$.
By the above (and by submultiplicativity) it follows that
\[\Vert \psi_n(gh)-\psi_n(g)\psi_n(h)\Vert=\Vert \xi_n(gh)-\xi_n(g)\xi_n(h)\Vert+o_\uf(\defect(\ph_n))\]
so it suffices to show that
\[\Vert \xi_n(gh)-\xi_n(g)\xi_n(h)\Vert=o_\uf(\defect(\ph_n))\]
which amounts to the following calculations
\begin{align*}
&\xi_n(gh)-\xi_n(g)\xi_n(h)\\
&=\tilde\ph_n(gh)-\tilde\ph_n(g)\tilde\ph_n(h)
\\&\ +\defect(\ph_n)(-\beta_n(gh)\tilde\ph_n(gh)+\tilde\ph_n(g)\beta_n(h)\tilde\ph_n(h)
+\beta_n(g)\tilde\ph_n(g)\tilde\ph_n(h))
\\&\ -\defect(\ph_n)^2\beta_n(g)\tilde\ph_n(g)\beta_n(h)\tilde\ph(h)
\\&=\defect(\ph_n)(-c_n(g,h)
\\&\ +\tilde\ph_n(g)\beta_n(h)\tilde\ph_n(h)-\beta_n(gh)\tilde\ph_n(gh)+\beta_n(g)\tilde\ph_n(g)\tilde\ph_n(h))
\\&\ -\defect(\ph_n)^2\beta_n(g)\tilde\ph_n(g)\beta_n(h)\tilde\ph_n(h).
\end{align*}
By Equation \eqref{tre} and the fact that submultiplicativity of the norm implies that $\lVert \beta_n(g)\tilde\ph_n(g)\beta_n(h)\tilde\ph_n(h)\rVert$ is bounded, this finishes the proof.
\end{proof}

At last we define the asymptotic representation $\ph'_n\colon S\to{\rm U}(k_n)$ by $\ph'_n=\psi_n|_S$ and reach the desired conclusion $\defect(\ph_n')=o_\uf(\defect(\ph_n))$. Let us, for reference's sake, formulate the result properly.
\begin{theorem}\label{speed}
Let $\Gamma=\langle S\mid R\rangle$ be a finitely presented group and let $\ph_n\colon S\to{\rm U}(k_n)$ be an asymptotic representation with respect to a family of submultiplicative, unitarily invariant norms.
Assume that the associated 2-cocycle $\alpha=\alpha(\ph_n)$ is trivial in $H^2(\Gamma,\Matuf)$. Then there exists an asymptotic representation $\ph_n'\colon S\to{\rm U}(k_n)$ such that
\begin{enumerate}
\item[$(i)$] $\dist(\ph_n,\ph'_n)=O_\uf(\defect(\ph_n))$
and
\item[$(ii)$] $\defect(\ph'_n)=o_\uf(\defect(\ph_n))$.
\end{enumerate}
\end{theorem}
\begin{proof}
We adopt the above notation. Assertion (i) follows from Proposition \ref{closeto}; let $r=x_1x_2\cdots x_m\in R$ be written as a reduced word, where $x_j\in S\cup S^{-1},j=1,\ldots,m$. By iteration of Lemma \ref{lemspeed} (using that $\psi_n$ takes unitary values and that $\lVert\cdot\rVert$ is unitarily invariant), we see that
\begin{align*}
\Vert\ph'_n(r)-1_{k_n}\Vert&=\Vert\psi_n(x_1)\psi_n(x_2)\cdots\psi_n(x_m)-1_{k_n}\Vert
\\&=\Vert\psi_n(x_1x_2)\psi_n(x_3)\cdots\psi_n(x_m)-1_{k_n}\Vert+o_\uf(\defect(\ph_n))
\\&\ \vdots
\\&=\Vert \psi(1_\Gamma)-1_{k_n}\Vert+o_\uf(\defect(\ph_n)).
\end{align*}
Since $\psi(1_\Gamma)=1_{k_n}$, we are done.
\end{proof}

The converse of Theorem \ref{speed} is also valid in the following sense.
\begin{prop}\label{omvendtverden}
Let $\Gamma=\langle S\mid R\rangle$ be a finitely presented group,
let $\ph_n,\psi_n: S\to {\mathrm U}(k_n)$ be asymptotic representations with respect to some family of submultiplicative, unitarily invariant norms and suppose
\begin{itemize}
\item $\dist(\ph_n,\psi_n)=O_\uf(\defect(\ph_n))$
and
\item $\defect(\psi_n)=o_\uf(\defect(\ph_n))$.
\end{itemize}
Then, the 2-cocycle $\alpha$ associated to $(\varphi_n)_{n \in \mathbb N}$ is trivial in $H^2(\Gamma,\Matuf)$.
In particular, if $\ph_n$ is sufficiently close to a homomorphism, $\alpha$ is trivial.
\end{prop}
\begin{proof}
If $\defect(\ph_n)=0$ for most $n\in\NN$ there is nothing to prove, so let us assume this is not the case.
Let $\tilde\ph_n,\tilde\psi_n\colon \Gamma\to\Uni(k_n)$ be the induced maps we get by fixing a section $\Gamma\to \mathbb F_S$ as explained in the beginning of this section. We note that the sequences $\tilde\ph_n$ and $\tilde\psi_n$ induce the same map $\ph_\uf$ in the limit.
Define
\[\gamma_n(g)=\frac{\tilde\ph_n(g)-\tilde\psi_n(g)}{\defect(\ph_n)}\]
for $n$ with $\defect(\ph_n)>0$ and $\gamma_n(g)=0$ otherwise. By the first bullet in our assumptions, $\gamma_n$ is essentially bounded, so it defines an element $\gamma(g)\in\Matuf$.
If we prove that
\[c(g,h)=\ph_\uf(g)\gamma(h)-\gamma(gh)+\gamma(g)\ph_\uf(h),\]
it will follow easily that $\beta(g):=\gamma(g)\ph_\uf(g)^*$ will
satisfy $d\beta=\alpha$. First note that
it follows from the second bullet in our assumptions that for every $g,h\in\Gamma$
\[\Vert \tilde\psi_n(gh)-\tilde\psi_n(g)\tilde\psi_n(h)\Vert=o_\uf(\defect(\ph_n)),\]
thus
\begin{align*}
\defect(\ph_n)\cdot(\tilde\ph_n(g)&\gamma_n(h)-\gamma_n(gh)+\gamma_n(g)\tilde\psi_n(h))
\\&= \tilde\ph_n(g)\tilde\ph_n(h)-\tilde\ph_n(g)\tilde\psi_n(h)-\tilde\ph_n(gh)+\tilde\psi_n(gh)
\\&+\tilde\ph_n(g)\tilde\psi_n(h)-\tilde\psi_n(g)\tilde\psi_n(h)
\\&=\tilde\ph_n(g)\tilde\ph_n(h)-\tilde\ph_n(gh)+\tilde\psi_n(gh)-\tilde\psi_n(g)\tilde\psi_n(h)
\\&=\defect(\ph_n)\cdot c_n(g,h)+o_\uf(\defect(\ph_n)).
\end{align*}
Now the result follows by dividing by $\defect(\ph_n)$ (which is possible for most $n$) and taking the limit.
\end{proof}

It is now clear that we are in need of large classes of groups for which general vanishing results for the second cohomology with Banach or Hilbert space coefficients can be proven. This will be the subject of the next section.
But first let us mention an alternative approach that can be used to prove Theorem \ref{speed}.

\subsection{Asymptotic representations and extensions}\label{altsec}
As mentioned in Section \ref{sec:cohom}, the second cohomology characterizes extensions of $\Gamma$ with abelian kernel and that in this picture coboundaries correspond to splitting extensions.
Thus Theorem \ref{speed} and Corollary \ref{omvendtverden} show that finding the improved $\ph_n'$ is equivalent to finding a splitting for a certain extension.
The connection between asymptotic representations and extensions can be seen directly without going through the above computations, and this idea can actually be used to prove Theorem \ref{speed}.
Since this approach is very illustrative (it shows, for instance, very clearly what rôle submultiplicativity plays), we sketch the proof.

We retain the assumptions from Section \ref{maryassumption} and introduce some more notation.
Letting $\eps_n:=\defect(\ph_n)$, for $n\in\NN$, we define
\[{\rm N}(O_\uf(\eps_n))=\Big\{(u_n)_{n\in\NN}\in\prod_{n\in\NN}\Uni(k_n)\mid \lVert u_n-1_{k_n}\rVert=O_\uf(\eps_n)\Big\}\]
and
\[\Uni(O_\uf(\eps_n))=\prod_{n\in\NN} \Uni(k_n)/N(O_\uf(\eps_n)).\]
Similarly, we define ${\rm N}(o_\uf(\eps_n))$ and
$\Uni(o_\uf(\eps_n))$. 
We saw that the asymptotic representation $(\ph_n)_{n\in\NN}$ induces a homomorphism
$\ph_\uf:\Gamma\to \Uniuf$, but Lemma \ref{relfact} actually implies the existence of an induced homomorphism
\[\hat\ph_\uf\colon \Gamma\to \Uni(O_\uf(\eps_n)).\]
Now we observe that the existence of $\ph_n'\colon S\to\Uni(k_n)$ with $\dist(\ph_n,\ph'_n)=O_\uf(\defect(\ph_n))$ and $\defect(\ph_n')=o_\uf(\defect(\ph_n))$ as in Theorem \ref{speed} is equivalent to the existence of a lift $\ph_\uf'$ of $\hat\ph_\uf$:
 \begin{equation*}\label{diagram1}
\xymatrix{         & \Uni(o_\uf(\eps_n)) \ar@{->}[dd]\\
    \Gamma \ar@{->}^{\ph_\uf'}[ru]\ar@{->}^{\hat\ph_\uf}[rd] & \\
                    & \Uni(O_\uf(\eps_n))
         }  
\end{equation*} 
We also see that the map $\hat\ph_\uf$ fits into the following commutative diagram
\begin{equation*}\label{diagram3}
\xymatrix{         
1\ar@{->}[r] & N\ar@{->}[r] & \Uni(o_\uf(\eps_n))\ar@{->}^\psi[r]  & \Uni(O_\uf(\eps_n)) \ar@{->}[r] & 1 \\
1\ar@{->}[r] & N\ar@{->}[r]\ar@{=}[u] & \hat\Gamma\ar@{->}[r]\ar@{->}[u]       & \Gamma\ar@{->}^{\hat\ph_\uf}[u]\ar@{->}[r] & 1
         }  
\end{equation*}
where $\hat\Gamma$ is the pullback through $\hat\ph$ and $\psi$ and 
\[N:={\rm N}(O_\uf(\eps_n))/{\rm N}(o_\uf(\eps_n)).\]
Combining these two observations, it easily follows that $\ph_n$ can be improved to $\ph_n'$ if and only if the bottom row in the latter diagram splits.
Now, since $\lVert\cdot\rVert$ is submultiplicative the group $N$ is actually abelian. Indeed, for all $T,S\in\Uni(k)$, we have that
\begin{align*}
\lVert TST^*S^*-1_{k}\rVert&=\lVert TS-ST\rVert
\\&=\lVert (T-1_{k})(S-1_k)-(S-1_k)(T-1_k)\rVert
\\&\leq 2\lVert T-1_k\rVert\lVert S-1_k\rVert,
\end{align*}
so if $(T_n)_{n\in\NN},(S_n)_{n\in\NN}\in N(O(\eps_n))$ then 
\[(T_nS_nT_n^*S_n^*)_{n\in\NN}\in N(O(\eps_n^2))\del N(o(\eps_n)),\]
whence the claim follows.
Hence, as explained in Section \ref{sec:cohom}, the extension $1\to N\to\hat\Gamma\to\Gamma\to1$ corresponds to an element $[\hat\alpha]\in H^2(\Gamma,N)$,
and we conclude that $\ph_n$ can be improved if and only if $[\hat\alpha]=[0]$.
Now, the coefficients $N$ are not exactly the same as $\Matuf$ in Theorem \ref{speed}, but with a little effort, one can prove that $N$ is a real Banach space (or a real Hilbert space in the case $\lVert\cdot\rVert=\frob\cdot$) with an isometric $\Gamma$-action and the existence of an equivariant homomorphism $\theta\colon N\to\Matuf$ such that $[\theta\circ\hat\alpha]=[\alpha]$.

\begin{rem}
We note that this approach also works for the most part if $\lVert\cdot\rVert$ is not submultiplicative. In this case, however, the group $N$ is not abelian and the second cohomology with non-abelian coefficients is much less tractable in general.
\end{rem}

This alternative approach to the problem at hand is rather conceptual and elegant, but also the proof that we chose to present in detail has its merits. The cocycle $\alpha$ can be computed directly from $(\ph_n)_{n\in\NN}$, and in cases where the associated 1-cochain $\beta$ can be computed explicitly from $\alpha$, this gives us an explicit expression for $\ph_n'$.

\section{Cohomology vanishing and examples of $n$-Kazhdan groups} \label{hika}
Recall that if $\Gamma$ is a finitely (or, more generally, compactly) generated group, then $\Gamma$ has Kazhdan's Property (T) if and only if the first cohomology $H^1(\Gamma,\mathcal H_\pi)=0$ for every unitary representation $\pi\colon \Gamma\to\Uni(\mathcal H_\pi)$ on a Hilbert space $\mathcal H_\pi$, see \cite{MR2415834} for a proof and more background information.
We will consider groups for which the higher cohomology groups vanish. Higher dimensional vanishing phenomena have been studied in various articles, see for example \cites{MR3284391, MR1721403, MR1465598
,MR1762517, MR1946553, MR3343347, MR1651383}.

We propose the following terminology.
\begin{defi}
Let $n\in\NN$. A group $\Gamma$ is called \emph{$n$-Kazhdan} if $H^n(\Gamma,\mathcal H_\pi)$ vanishes for all unitary representations $(\pi,\mathcal H_\pi)$ of $\Gamma$. We call $\Gamma$ \emph{strongly $n$-Kazhdan}, if $\Gamma$ is $k$-Kazhdan for $k=1,\ldots, n$.
\end{defi}
So 1-Kazhdan is the  Kazhdan's classical Property (T).
See \cite{MR3284391, MR3343347} for discussions of other related higher dimensional analogues of Property (T). It will be central in our proof that by an application of the open mapping theorem, vanishing of cohomology with Hilbert space coefficients implies that cocycles are coboundaries with control on the norms. This is explained in the following proposition and its corollary, where we use the terminology introduced in Equation \eqref{eqdefnorm}.

\begin{prop}\label{vanishprop}
Let $n\in\NN$, let $\Gamma$ be a countable group, let $\pi\colon \Gamma\to\Uni(\mathcal H_\pi)$ be a unitary representation, and assume that $H^n(\Gamma,\mathcal H_\pi)=\{0\}$. Then for every finite set $F\del \Gamma^{n-1}$  there exist a finite set $F_\pi\del \Gamma^{n}$ and a constant $C_{\pi,F}\geq 0$ such that for every cocycle $z\in Z^{n}(\Gamma,\mathcal H_\pi)$ there is an element $b\in C^{n-1}(\Gamma,\mathcal H_\pi)$ such that $z=d^{n-1} b$ and $\Vert b\Vert_{F} < C_{\pi,F}\Vert z\Vert_{F_\pi}$.
\end{prop}
\begin{proof}
By definition of the topology on $C^n(\Gamma,\mathcal H_\pi)$, the basic open sets are given by
\[U_{\delta,F'}=\{f\in C^n(\Gamma,\mathcal H_\pi)\mid \Vert f\Vert_{F'}<\delta\},\]
for a finite $F'\del\Gamma^n$ and $\delta>0$.
Since the map $d^{n-1}\colon C^{n-1}(\Gamma,\mathcal H_\pi)\to Z^n(\Gamma,\mathcal H_\pi)$ is linear, bounded and surjective, the open mapping theorem applies (see \cite{MR1157815}), so there are $C_{\pi,F}>0$ and $F_\pi\del \Gamma^n$ such that
\[U_{C_{\pi,F}^{-1},F_\pi}\cap Z^{n}(\Gamma,\mathcal H_\pi)\del d^{n-1}(U_{1,F}).\] 
In other words, if $z\in Z^n(\Gamma,\mathcal H_\pi),\Vert z\Vert_{F_\pi}=1$, then $C_{\pi,F}^{-1} z\in U_{C_{\pi,F}^{-1},F_\pi}$, so there is
$b\in C^{n-1}(\Gamma,\mathcal H_\pi)$ such that $d^{n-1} b=z$ and $\Vert b\Vert_F<C_{\pi,F}= C_{\pi,F}\Vert z\Vert$. This proves the claim.
\end{proof}

We need the fact that if $H^2(\Gamma,\mathcal H_\pi)$ vanishes universally, the set $F_\pi$ and the bound $C_{\pi,F}$ can be chosen universally for all unitary representations $\pi$. This is the consequence of an easy diagonalisation argument.

\begin{cor}\label{vanishcor}
Let $n\in\NN$ and $\Gamma$ be a countable $n$-Kazhdan group. Then for every finite set $F\del \Gamma^n$ there are a finite set $F_0\del \Gamma^{n-1}$ and a constant $C_F\geq 0$ such that for all unitary representations $\pi$ of $\Gamma$ and all cocycles $z\in Z^n(\Gamma,\mathcal H_\pi)$ there is an element $b\in C^{n-1}(\Gamma,\mathcal H_\pi)$ such that $z=d^{n-1} b$ and $\Vert b\Vert_{F}< C_F\Vert z\Vert_{F_0}$.
\end{cor}
%\begin{proof}
%Consider a family $(\pi_i)_{i\in I}$ of unitary representations of $\Gamma$. Then $C\geq 0$ and $F'\del \Gamma$ satisfy the conclusion of Proposition \ref{vanishprop} for all $\pi=\pi_i$ if and only if $C$ and $F'$  also satisfy the conclusion for $\pi=\bigoplus_{i\in I}\pi_i$. Since every unitary representation is the direct sum of cyclic representations, it is thus enough to find $C$ and $F'$ that apply to all cyclic rerpesentations.
%
%For this, recall that, if we consider unitary representations up to unitary equivalence, the cyclic representations of $\Gamma$ form a set. Indeed, every cyclic representation is, up to unitary equivalence, the GNS-construction of a function of positive type on $\Gamma$ and these functions form a set. Thus we can form a \emph{universal representation} $\pi_u$ of $\Gamma$ as the direct sum of all cyclic representations (that is, one representative in each equivalence class).
%Since every cyclic representation is isomorphic to a direct summand of the universal represenation $\pi_u$ of $\Gamma$, the result follows from Proposition \ref{vanishprop} applied to $\pi_u$.
%\end{proof}

We also observe the following extension proposition.

\begin{prop}\label{kazhext}
Consider a short exact sequence of groups.
\[1\rightarrow \Lambda\rightarrow \tilde\Gamma \rightarrow \Gamma\rightarrow 1.\]
If $\Lambda$ is strongly $n$-Kazhdan and $\Gamma$ is $n$-Kazhdan, then $\tilde\Gamma$ is also $n$-Kazhdan. In particular, this applies if $\Lambda$ or $\Gamma$ is finite.
\end{prop}
\begin{proof}
By the Hochschild-Serre spectral sequence \cite{MR1324339}, it is enough to show that $H^k(\Gamma,H^l(\Lambda,\mathcal H_{\pi|_\Lambda}))$ vanishes for all $k,l\in\NN$ with $k+l= n$.
If $l>0$, then $H^l(\Lambda,\mathcal H_{\pi|_\Lambda})$ vanishes. For $l=0$ and $k=n$, we have $H^0(\Lambda,\mathcal H_{\pi|_\Lambda})=\mathrm{Fix}(\pi|_\Lambda)$ (the set of fixed vectors in $\mathcal H_{\pi|_\Lambda}$), which is a Hilbert space, and the induced action of $\Gamma$ is a unitary representation, so we conclude that $H^n(\Gamma,H^0(\Lambda,\mathcal H_{\pi|_\Lambda}))$ vanishes.
\end{proof}

In view of the previous section, it is natural to ask if there exists a non-residually finite group, such that $H^2(\Gamma,A)$ vanishes for all $C^*$-algebras $A$ equipped with an action of $\Gamma$ by automorphisms. We are not able to answer this questions, however, one can show that $H^1(\Gamma,\ell^{\infty}(\Gamma))$ does not vanish for any infinite group, which makes a positive answer somewhat unlikely. Here, we view $\ell^{\infty}(G)$ as a $G$-module with respect to the right translation action. Indeed, let $d \colon \Gamma \times \Gamma \to \NN$ be a proper left-invariant metric. Then, $c(g):= (h \mapsto d(1_\Gamma,h) - d(1_{\Gamma},hg))_{h \in \Gamma}$ defines a cocycle  $c \colon \Gamma \to \ell^{\infty}(\Gamma)$ which cannot be the boundary of an element in $\ell^{\infty}(\Gamma)$ if $\Gamma$ is infinite.

\subsection{Higher rank $p$-adic lattices are $2$-Kazhdan} \label{sec:hk}

Finally, this section provides examples, or every $n\geq 2$, of groups which are $n$-Kazhdan. The results are essentially known and we recall them in detail for convenience.

Let $K$ be a non-archimedean local field of residue class $q$, i.e.\ if $\mathcal O \subset K$ is the ring of integers and $\mathfrak m \subset \mathcal O$ is its unique maximal ideal, then $q= |\mathcal O/\mathfrak m|$. Let $\mathbf{G}$ be a simple $K$-algebraic group of $K$-rank $r$ and assume that $r \geq 1$. The group $G:= \mathbf G(K)$ acts on the associated Bruhat-Tits building $\mathcal B$. For more information on the theory of buildings, see \cite{MR2439729}. The latter is an infinite, contractible, pure simplicial complex of dimension $r$, on which $G$ acts transitively on the chambers, i.e.\ the top-dimensional simplices. Let $\Gamma$ be a uniform lattice in $G$, i.e. a discrete cocompact subgroup of $G$. When $\Gamma$ is also torsion free (which can always be achieved by replacing $\Gamma$ by a finite index subgroup), then the quotient $X:= \Gamma \backslash \mathcal B$ is a finite $r$-dimensional simplicial complex and $\Gamma= \pi_1(X).$ In particular, the group $\Gamma$ is finitely presented. We will use the following theorem which essentially appears in work of Ballmann and \'Swi\polhk atkowski \cite{MR1465598} building on previous work of Garland \cite{MR0320180}.

\begin{theorem} \label{higherkazhdan}
For every natural number $r \geq 2$, there exists $q_0(r) \in \mathbb N$ such that the following holds. If $q \geq q_0(r)$ and $G$ and $\Gamma$ are as above, then $\Gamma$ is strongly $(r-1)$-Kazhdan. In particular, if $r\geq 3$, then $\Gamma$ is $2$-Kazhdan.
\end{theorem}

Recall, that being $1$-Kazhdan is equivalent to Kazhdan's property (T). As it is well known, $G$ and $\Gamma$ as above have property (T) for every $r \geq 2$ and for \emph{all} $q$. It is quite plausible that this is also true for in the context of the preceding theorem.

Note that such $\Gamma$ contains a finite-index torsion free group $\Lambda$. Proposition \ref{kazhext} implies that it suffices to prove that $\Lambda$ is $(r-1)$-Kazhdan. So one can assume that $\Gamma$ is torsion free.

Theorem \ref{higherkazhdan} for finite dimensional Hilbert spaces is Theorem 8.3 in the seminal paper of Garland \cite{MR0320180}. The general case is stated in the last paragraph of Section 3.1 on page 631 in the work of Ballmann-\'Swi\polhk atkowski \cite{MR1465598}. It is deduced from Theorem 2.5 there: that Theorem asserts \'a la Garland  \cite{MR0320180} that the desired cohomology vanishing follows from sharp estimates of the spectral gap of the local Laplacians, i.e. the Laplacians of the proper links of the complex. These estimates (called also \emph{p-adic curvature})  
are given in Lemma 6.3 and Lemma 8.2 in \cite{MR0320180}. So altogether Theorem \ref{higherkazhdan} is proven. The method and estimates of Garland are used also in \cite{MR1651383, MR1995802} and more recently \cite{MR3557457, MR3343347}.

Let us give the reader just a notational warning:  when we say rank (following the common practice nowadays) we mean the  $K$-rank of $G$ as a $p$-adic group (and we denoted it by $r$) and then it follows that the dimension of the associated Bruhat-Tits building is equal to $r$. Garland refers to the rank of the Tits system which in his notation he denotes $l+1$. Hence, our $r$ is equal to his $l$. 

%Theorem \ref{higherkazhdan} is essentially known or can be easily deduced from results in the literature, see \cite[Theorem 2.5]{MR1465598}, \cite[Lemma 6.3]{MR0320180} and the remarks at the end of Section 3.1 in \cite{MR1465598}. The proof is based on the seminal work of Garland \cite{MR0320180}, who showed vanishing of the cohomology for such lattices w.r.t.\ finite dimensional (not necessarily unitary) representations. His method was extended and applied in various directions by \cite{MR1995802}, \cite{MR1651383} and \cite{MR1465598}.

\vspace{0.1cm}

It is very natural to wonder what happens in the analogous real case. It is worth noting that already $H^5({\rm SL}_n(\mathbb Z),\mathbb R)$ is non-trivial for $n$ large enough \cite{borel}; thus ${\rm SL}_n(\mathbb Z)$ fails to be $5$-Kazhdan for $n$ large enough. 
Similarly, note that $H^2({\rm Sp}(2n,\mathbb Z),\mathbb R)=\mathbb R$ for all $n \geq 2$  \cite{borel}, so that the natural generalization to higher rank lattices in real Lie groups has to be formulated carefully; maybe just by excluding an explicit list of finite-dimensional unitary representations.

\begin{question}
Is ${\rm SL}_n(\mathbb Z)$ $2$-Kazhdan (at least for large $n$)?
\end{question}

\section{Proofs of the main results}

In order to finish the proofs of Theorem \ref{main} and Theorem \ref{main2}, we need to show that finitely presented $2$-Kazhdan groups are Frobenius-stable and that some of them are not residually finite. The main result follows then from Corollary \ref{corka} and the constructions in Section \ref{kazhres}.

\subsection{The Frobenius-stability of $2$-Kazhdan groups.}
We now consider 2-Kazhdan groups and asymptotic representations with respect to the Frobenius norm. As $\prod_{n\to\uf}( {\rm M}_{k_n}(\mathbb C),\frob{\cdot})$ is a Hilbert space, the techniques of the Section \ref{sec:def} can be applied and the defect of every asymptotic representation can be diminished. We start by completing the proof of Theorem \ref{main2}.

\begin{theorem}\label{kazh2stab}
Let $\Gamma$ be a finitely presented group. If $\Gamma$ is 2-Kazhdan, then it is Frobenius-stable.
\end{theorem}
\begin{proof}
Let $\Gamma=\langle S\mid R\rangle$.
As mentioned, the ultraproduct \[\mathrm M_{\uf}^{\frob{\cdot}}:=\prod_{n\to\uf}( {\rm M}_{k_n}(\mathbb C),\frob{\cdot})\] is a Hilbert space and $\Gamma$ acts on this space by invertible isometries, i.e.\ by unitaries, so $H^2(\Gamma,\mathrm M_{\uf}^{\frob{\cdot}})$ vanishes.
By Corollary \ref{vanishcor} together with the bounds from Equation (\ref{cbound}) there is a constant $C$ such that for all asymptotic representations $\ph_n\colon \Gamma\to {\rm U}(k_n)$ with respect to $\frob{\cdot}$, we can choose the associated 1-cochain $\beta$ so that it satisfies
\[2\max_{s\in S}\frob{\beta(s)}\leq C.\]
Define the quantity \[\theta(\ph):=\homdist(\ph)-2C\defect(\ph)\] for any map $\ph\colon S\to\Uni(k)$ (for any $k\in\NN$). We note that if $\ph_n:S\to{\rm U}(k_n)$ is any asymptotic representation, then $\lim_{n\to\uf}\theta(\ph_n)\geq 0$ and equality holds if and only if $\ph_n$ is equivalent to a sequence of homomorphisms.

Now fix a sequence $(\eps_n)_{n\in\NN}$ of strictly positive real numbers such that $\lim_{n\to\uf}\eps_n=0$ and let $(k_n)_{n\in\NN}$ a sequence of natural numbers. By the above, we need to prove that for all sequences of $\eps_n$-almost representations $\psi_n\colon S\to{\rm U}(k_n)$ the quantity $\theta(\psi_n)$ tends to $0$. For each $n\in\NN$, the set of $\eps_n$-almost homomorphisms $\ph\colon S\to{\rm U}(k_n)$ is compact and since $\theta$ is continuous, there is $\ph_n\colon S\to {\rm U}(k_n)$ such that $\defect(\ph_n)\leq \eps_n$ and $\ph_n$ maximizes $\theta$ for all $n\in\NN$. Evidently $\ph_n$ is an asymptotic representation.
Thus, by Proposition \ref{closeto} and Theorem \ref{speed} there is an asymptotic representation $\ph_n'\colon S\to{\rm U}(k_n)$ such that $\dist(\ph_n,\ph_n')\leq C\defect(\ph_n)$ and
\begin{equation}\label{howi}
\defect(\ph_n')\leq \frac{1}{4}\defect(\ph_n)
\end{equation}
for most $n\in\NN$.
In particular, $\ph'_n$ is also an $\eps_n$-almost representation, and it follows that for most $n$, we have
\[\homdist(\ph_n)\leq \homdist(\ph'_n)+C\defect(\ph_n).\]
Furthermore, by maximality we have that
\begin{equation*}
\homdist(\ph'_n)-2C\defect(\ph'_n)=\theta(\ph'_n)\leq \theta(\ph_n)=\homdist(\ph_n)-2C\defect(\ph_n),
\end{equation*}
and putting these estimates together, we get
\begin{equation}\label{howo}
\homdist(\ph'_n)-2C\defect(\ph'_n)\leq \homdist(\ph'_n)-C\defect(\ph_n),
\end{equation}
or
\[\defect(\ph_n)\stackrel{\eqref{howo}}{\leq} 2\defect(\ph'_n)\stackrel{\eqref{howi}}{\leq} \frac{1}{2}\defect(\ph_n),\]
which can only be the case if $\defect(\ph_n)=0$ for most $n$.
But then $\ph_n$ is really a representation for most $n\in\NN$, so $\homdist(\ph_n)=0$ and we conclude $\lim_{n\to\uf}\theta(\ph_n)=0$. Since $\theta(\ph_n)$ was chosen maximal, we conclude that $\lim_{n\to\uf}\theta(\psi_n)=0$ for all $\eps_n$-almost representations $\psi_n$.
\end{proof}

\begin{rem}
Note that the same proof is still valid if one replaces $\frob{\cdot}$ with any submultiplicative norm $\lVert\cdot\rVert$ and the $2$-Kazhdan assumption with a suitable cohomology vanishing assumption. This, for instance, gives a sufficient condition for stability with respect to the operator norm, where one could assume vanishing of second cohomology with coefficients in a $C^*$-algebra, but it seems difficult to prove the existence of a group $\Gamma$ with such properties -- a task that will already occupy the remaining sections in the Hilbert space case.
\end{rem}

\begin{rem}
Note that Theorem \ref{kazh2stab} together with Proposition \ref{kazhext} imply that virtually free groups are Frobenius-stable -- a fact that seems cumbersome to establish directly.
\end{rem}

For the sake of reference, we formulate following dichotomy, which is an immediate corollary to Theorem \ref{kazh2stab}, explicitly.

\begin{cor} \label{corka}
Let $\Gamma$ be a finitely presented $2$-Kazhdan group. Then either
\begin{itemize}
\item $\Gamma$ is residually finite, or
\item $\Gamma$ is not Frobenius-approximated.
\end{itemize}
\end{cor}

The techniques in Section \ref{sec:def} rely on submultiplicativity of the norm and thus cannot be directly applied to the normalized Hilbert-Schmidt norm $\hsnorm\cdot$. It is worth noting, though, that since $\frac{1}{\sqrt{k}}\frob{A}=\hsnorm A\leq \opnorm{A}\leq \frob{A}$ for $A\in {\rm M}_k(\mathbb C)$, we get the following immediate corollary to Theorem \ref{kazh2stab}.

\begin{cor}\label{kazh2hyper}
Let $\Gamma=\langle S\mid R\rangle$ be a finitely presented 2-Kazhdan group and let $\ph_n\colon S\to{\rm U}(k_n)$ be a sequence of maps such that
\[\defect(\ph_n)=o_\uf(k_n^{-1/2}),\]
where the defect is measured with respect to either $\hsnorm\cdot$ or $\opnorm\cdot$. Then
$\ph_n$ is equivalent to a sequence of homomorphisms.
\end{cor}
\begin{proof}
Let $\lVert\cdot\rVert$ be the norm in question. Then 
\[\defect_{\frob\cdot}(\ph_n)\leq \sqrt{k_n}\defect_{\lVert\cdot\rVert}(\ph_n)=o_\uf(1),\]
in other words, $\ph_n$ is an asymptotic representation with respect to $\frob\cdot$ so by Theorem \ref{kazh2stab} there are representations $\pi_n\colon\Gamma \to{\rm U}(k_n)$ with
\[\lVert \ph_n(s)-\pi_n(s)\rVert\leq \frob{\ph_n(s)-\pi_n(s)}=o_\uf(1)\]
for $s\in S$.
\end{proof}

The preceding corollary provides some quantitative information on the Connes Embedding Problem. Indeed, if a finitely presented, non-residually finite, 2-Kazhdan group is $\Uni_{\mathrm{HS}}$-approximated, then there is some upper bound on the quality of the approximation in terms of the dimension of the unitary group. Needless to say it would be very interesting to decide if groups as above are $\Uni_{\mathrm{HS}}$-approximated.

\subsection{Non-residually finite $2$-Kazhdan groups} \label{kazhres}
In this section, we present examples of finitely presented, non-residually finite groups which are $2$-Kazhdan and hence finish the proof of Theorem \ref{main}. Note first that all the examples $\Gamma$ presented in Section \ref{sec:hk} are residually finite. In this section we will show that some of these $\Gamma$'s have finite central extensions
$$1 \to C \to \tilde \Gamma \to \Gamma \to 1,$$
where $C$ is a finite cyclic group and $\tilde \Gamma$ is not residually finite. Now, $C$ being finite is strongly $n$-Kazhdan for every $n$ and so, if $\Gamma$ is $2$-Kazhdan, then the same holds for $\tilde \Gamma$ by Proposition \ref{kazhext}. Hence, we may combine our results of this section with the results from the previous section to obtain examples of $2$-Kazhdan groups which are not residually finite.

Our construction will imitate the construction of Deligne \cite{MR507760} of non-residually finite central extensions of some non-uniform arithmetic lattices in real Lie groups. See also the work of Raghunathan \cite{MR735524, MR1355004}, where such central extensions were constructed for some uniform lattices in ${\rm Spin}(2,n)$. These examples were later used by Toledo \cite{MR1249171} in his famous work showing the existence of algebraic varieties with non-residually finite fundamental groups. A short and very readable exposition of Deligne's argument was given by Witte-Morris \cite{morris}. 

Our examples are $p$-adic analogues of Deligne's examples and his original proof actually works for them. He assumed that the algebraic group $\mathbf G$ to be isotropic and hence got only non-uniform lattices, as at the time the congruence subgroup property was known only in such cases. Nowadays, we can argue for more general lattices along the same lines.

Let $D$ be the standard quaternion algebra over $\mathbb Z$, defined as $$D= \mathbb Z\langle i,j,k \rangle/(i^2=j^2=k^2=-1, ij=k)$$ and set $D_{R}:=R \otimes_{\mathbb Z} D$ for an arbitrary unital commutative ring $R$. It is well-known that $D_{\mathbb R}$ is the Hamiltonian division algebra $\mathbb H$, whereas
$D_{{\mathbb Q}_p} \cong {\rm M}_2({\mathbb Q}_p)$ for $p\geq 3$, where the second isomorphism is basically a consequence of the fact that the congruence $x^2+y^2=-1$ can be solved modulo $p$.
Consider also the standard involution $\tau \colon D_R \to D_R$ and let $h \colon D_R^n \times D_R^n \to D_R$ be the canonical sesqui-linear hermitian form on $D_R^n$, i.e.\ $$h((x_1,\dots,x_n),(y_1,\dots,y_n))= x_1\tau(y_1) + \cdots + x_n\tau(y_n).$$ Consider now $\mathbf G(R) := {\rm SU}(n,D_R,h)$. Note that $\mathbf G(R)$ is simply the group formed by those $n \times n$-matrices with entries in $D_R$, such that the associated $D_R$-linear map preserves the form $h$. The functor $\mathbf G$ is an absolutely almost simple, simply connected $\mathbb Q$-algebraic group which is $\bar {\mathbb Q}$-isomorphic to ${\mathbf {Sp}}(2n)$ and hence of type $C_n$, see $\S 2.3$ in \cite{MR1278263}. 
Embedding $D_{\mathbb R} \subset M_2(\mathbb C)$, one can show that ${\mathbf G}(\mathbb R)$ is isomorphic to a simply connected compact Lie group of type $C_n$, namely the quaternionic unitary group ${\rm Sp}(n)={\rm U}(2n) \cap {\rm Sp}(2n,\mathbb C).$ 
%For all this and more background information see \cite{MR1278263}.

Let now $p \geq 3$ be a rational prime. Since $D(\mathbb Q_p)\cong {\rm M}_2(\mathbb Q_p)$, the group $\mathbf G$ becomes split over $\mathbb Q_p$ and $\mathbf G(\mathbb Q_p)$ is a non-compact group isomorphic to ${\rm Sp}(2n,\mathbb Q_p)$. The group $\Gamma:= \mathbf{G}(\mathbb Z[1/p])$ sits diagonally as a lattice in $\mathbf{G}(\mathbb R) \times \mathbf{G}(\mathbb Q_p)$. However, since $\mathbf{G}(\mathbb R)$ is compact, this yields that $$\Gamma=\mathbf{G}(\mathbb Z[1/p]) \subset \mathbf{G}(\mathbb Q_p)$$ is also a lattice. It is a standard fact that lattices in ${\rm Sp}(2n,\mathbb Q_p)$ are cocompact, basically since ${\rm Sp}(2n,\mathbb Q_p)$ admits a basis of neighborhoods of the identity that consists of torsion free subgroups.
In this concrete case, we can identify $\Gamma$ with the group
$${\rm U}(2n) \cap {\rm Sp}(2n,\mathbb Z[i, 1/p]).$$

It was proved by Rapinchuk \cite{MR1015345} and Tomanov \cite{MR1022796} that the group $\Gamma=\mathbf{G}(\mathbb Z[1/p])$ has the congruence subgroup property. Let us explain what this means in the adelic language: The group $\Gamma$ is a subgroup of $\mathbf G(\mathbb Q)$ and we can define two topologies on ${\mathbf G}(\mathbb Q)$ as follows. The first is the \emph{arithmetic topology}, for which the arithmetic subgroups, i.e. the subgroups commensurable to $\Gamma$ serve as a fundamental system of neighborhoods of the identity. The second is the \emph{congruence topology} for which we take as a basis of neighborhoods of the identity only those arithmetic groups which contain, for some natural number $m$ with $(m,p)=1$, one of the principal congruence subgroups
$$\Gamma(m) := {\rm ker}\left(\mathbf G(\mathbb Z[1/p])) \to \mathbf G(\mathbb Z/m\mathbb Z) \right).$$

We denote by $\widehat{\mathbf G(\mathbb Q)}$ the completion with respect to the arithmetic topology and by $\overline{\mathbf G(\mathbb Q)}$ the completion with respect to the congruence topology. There is a canonical surjective homomorphism $$\pi \colon \widehat{\mathbf G(\mathbb Q)} \to \overline{\mathbf G(\mathbb Q)}.$$ The result of Rapinchuk and Tomanov \cite{MR1015345, MR1022796} combined with the work of Prasad-Rapinchuk \cite{MR1441007} says that in our case, $\pi$ is an isomorphism of topological groups.

Now, by the strong approximation theorem, $\overline{\mathbf G(\mathbb Q)}$ is isomorphic to
$$\mathbf G\left(\mathbb A_{f \setminus \{p\}}\right) = \prod_{l \neq p}\!{}^*\ \mathbf G(\mathbb Q_l),$$
where $\prod^*$ denotes the restricted product as usual and $\mathbb A_{f \setminus \{p\}}$ is a subring of the $\mathbb Q$-adeles $\mathbb A$, the restricted product of $\mathbb Q_l$ for all primes $l \neq p$. In particular, we get
$$\mathbf G(\mathbb A)= \mathbf G(\mathbb R) \times \mathbf G(\mathbb Q_p) \times \overline{\mathbf G(\mathbb Q)}.$$

Now a result of Prasad \cite{MR2020658} (see also Deodhar \cite{MR0372057} and Deligne \cite{MR507760}) says that for every $p$, $\mathbf G(\mathbb Q_p)$ has a universal central extension
$$1 \to C(p) \to  \widetilde{\mathbf G(\mathbb Q_p)} \to \mathbf G(\mathbb Q_p) \to 1,$$
where $C(p)$ denotes the group of roots of unity in $\mathbb Q_p$, i.e.\ a cyclic group of order $p-1$. We denote by $\widetilde \Gamma$ and by $\widetilde{\mathbf G(\mathbb Q)}$ the inverse images of $\Gamma$ and $\mathbf G(\mathbb Q)$ under the quotient map in the above extension. 

\vspace{0.1cm}

We claim that if $p \geq 5$, then the group $\widetilde{\Gamma}$ is not residually finite.
\begin{prop} Every finite index subgroup of $\widetilde{\Gamma}$ contains the unique subgroup of index $2$ in $C(p)$. In particular, if $p \geq 5$,  $\widetilde{\Gamma}$ is not residually finite.
\end{prop}
\begin{proof}
To prove this, we will lift the arithmetic topology from $\mathbf G(\mathbb Q)$ to its central extension $\widetilde{\mathbf G(\mathbb Q)}$ as follows. We define the arithmetic topology on $\widetilde{\mathbf G(\mathbb Q)}$ as the topology for which all subgroups commensurable to $\widetilde{\Gamma}$ serve as a fundamental system of neighborhoods of the identity. We denote by $\widehat{\widetilde{\mathbf G(\mathbb Q)}}$ its Hausdorff completion. It is clear from the definition that there exists a central extension of topological groups
$$1 \to Z \to \widehat{\widetilde{\mathbf G(\mathbb Q)}} \to \widehat{\mathbf G(\mathbb Q)} \to 1,$$
where $Z$ is a quotient of $C(p)$, say by the quotient homomorphism $\mu \colon C(p) \to Z$, where $\ker(\mu)$ is exactly the intersection of all the finite-index subgroups of $\widetilde{\Gamma}$. The ultimate goal is to show that if $p \geq  5$, then $\ker(\mu)$ is non-trivial which would show that $\widetilde{\Gamma}$ is not residually finite. Define now
$$\widetilde{E} := \mathbf G(\mathbb R) \times \widetilde{\mathbf G(\mathbb Q_p)} \times \widehat{ \widetilde{\mathbf G(\mathbb Q)}}$$
and observe that it maps onto $\mathbf G(\mathbb R) \times {\mathbf G(\mathbb Q_p)} \times \widehat{{\mathbf G(\mathbb Q)}} = \mathbf G(\mathbb A)$ with kernel
$1 \times C(p) \times Z.$ Finally, we set
$$E= \frac{\widetilde{E}}{\{(1,a,b) \in 1 \times C(p) \times Z \mid b=\mu(a)\}}.$$
Now, the group $E$ is a central extension of $\mathbf G(\mathbb A)$ with kernel isomorphic to $Z$. Moreover, we also see from the definitions that the natural diagonal map $\widetilde{\mathbf G(\mathbb Q)} \to \widetilde E \to E$ sends $a \in C(p) \subset \widetilde{\mathbf G(\mathbb Q)}$ to  $(1,a,\mu(a))$ and hence factors through a homomorphism $\mathbf G(\mathbb Q) \to E$.
This shows that the central extension
$$1 \to Z \to E \to \mathbf G(\mathbb A) \to 1$$
splits over the subgroup $\mathbf G(\mathbb Q)$ of $\mathbf G(\mathbb A)$. Note that since $\mathbf G(\mathbb Q)$ is perfect, the same applies to $\mathbf G(\mathbb A)$. Then a  result on $\mathbf G(\mathbb A)$ going back to Moore \cite{MR0244258} for split groups and Prasad-Rapinchuk \cite{MR1441007} for the general case, asserts that the universal central extension of $\mathbf G(\mathbb A)$ that splits over $\mathbf G(\mathbb Q)$ has, in the case of our $\mathbf G$, a kernel of order $2$ -- basically since the groups of roots of unity in $\mathbb Q$ is $\{\pm 1\}$.
Hence, we can conclude that $|Z| \leq 2$. This proves the first part. More specifically, this shows that the kernel of the map from the profinite completion $\widehat{\widetilde{\Gamma}}$ of $\widetilde{\Gamma}$, which is realised as a compact-open subgroup of $\widehat{\widetilde{\mathbf G(\mathbb Q)}}$, to the profinite completion $\widehat \Gamma$ of $\Gamma$, which is realised as a compact-open subgroup of $\widehat{\mathbf G(\mathbb Q)} = \overline{\mathbf G(\mathbb Q)}$ is of order at most $2$. Hence, every finite index subgroup of $\widetilde{\Gamma}$ contains the index $2$ subgroup of $C(p)$ in the center of $\widetilde{\Gamma}$.  

Now, if $p \geq 5$, then $2<p-1$ and this proves that $\widetilde{\Gamma}$ is not residually finite. 
\end{proof}
In conclusion, since $\widetilde \Gamma$ is $2$-Kazhdan, by Theorem \ref{higherkazhdan} and Proposition \ref{kazhext}, it can not be Frobenius-approximated by Corollary \ref{corka}. This finishes the proof of Theorem \ref{main}.

\section*{Acknowledgments}
The first and the last named authors were supported by ERC Consolidator Grant No.\ 681207. The second named author was supported by the ERC when visiting the Hebrew University at Jerusalem. The third author was supported by the ERC, NSF and BSF. 

The last author wants to thank David Fisher for fruitful discussions back in March 2011 -- the idea of a cohomological obstruction to stability of asymptotic representations was first found there and later independently by the second and third author. We are grateful to Pierre Pansu and especially to Andrei  Rapinchuk for useful remarks and references. 
We thank the Isaac Newton Institute in Cambridge for its hospitality 
during the workshop \emph{Approximation, deformation, and quasification} (supported by EPSRC Grant No.  EP/K032208/1) as part of the program on \emph{Non-positive curvature: group actions and cohomology}.

\end{document}